\documentclass[11pt]{article}

\usepackage[a4paper, margin=2.3cm]{geometry}
\usepackage{amsmath,amssymb,amsthm}
\usepackage{verbatim}
\usepackage{color}
\usepackage{parskip}
\usepackage{hyperref}
\usepackage{cleveref}
\usepackage{setspace}
\usepackage{graphicx}
\usepackage{mathtools}
\usepackage[thinc]{esdiff}

\newtheorem{theorem}{Theorem}[section]
\newtheorem{remark}{Remark}[section]

\newtheorem{lemma}[theorem]{Lemma}

\newtheorem{conjecture}[theorem]{Conjecture}
\newtheorem*{definition*}{Definition}

\newcommand{\R}{{\mathbb R}}

\def\N{\mathbb{N}}
\def\Z{\mathbb{Z}}
\def\spt{\text{spt}}

\allowdisplaybreaks

\begin{document}
\title{Discretized sum-product type problems: Energy variants and Applications}
\author{Quy Pham \thanks{Department of Mathematics, University of Rochester. Email: \href{mailto:qpham3@ur.rochester.edu}{qpham3@ur.rochester.edu}}\and Thang Pham \thanks{University of Science, Vietnam National University, Hanoi. Email: \href{mailto:phamanhthang.vnu@gmail.com}{phamanhthang.vnu@gmail.com}}\and Chun-Yen Shen\thanks{Department of Mathematics, National Taiwan University. Email: \href{mailto:cyshen@math.ntu.edu.tw}{cyshen@math.ntu.edu.tw}}}
\date{}
\maketitle

\begin{abstract}
In this paper, we provide estimates for the additive discretized energy of 
\[\sum_{c\in C} |\{(a_1, a_2, b_1, b_2)\in A^2\times B^2: |(a_1 +cb_1) - (a_2 + cb_2)|\le \delta\}|_{\delta},\] that depend on non-concentration conditions of the sets. Our proof follows the Guth-Katz-Zahl approach (2021) with appropriate changes along the way clarifying and optimizing many of the steps. Several applications will also be discussed.

\bigskip

\end{abstract}

\tableofcontents
%\begin{abstract}
%Given $\delta\in 2^{-\mathbb{N}}$ and finite $\delta$-separated sets $A, B\subset \delta \mathbb{Z}\cap [0,1]$ and $C\subset [1/2, 1]$, we establish non-trivial estimates on the energy
%\[\sum_{c\in C} |\{(a_1, a_2, b_1, b_2)\in A^2\times B^2: |(a_1 +cb_1) - (a_2 + cb_2)|\le \delta\}|.\]
%To prove our main results, we introduce a number of novel approaches which make use of a combination of methods from both continuous and discrete settings including a pivoting argument, which has been used in the finite field setting due to Murphy and Petridis, the recent Guth-Katz-Zahl's method for the discretized sum-product problem and a Dabrowski-Orponen-Villa point-tube incidence bound. We give several applications of our results, proving explicit bounds on the size of the $\delta$-covering of sets $A + cB$ (for some $c\in C)$ and $C(A+A)$. 

%As a further application, for compact sets $A, B\subset \R$, we also prove new explicit upper bounds on the quantity $\dim_H\{c\in \mathbb{R}\colon \dim_H(A+cB)\le \alpha+\epsilon\}$, where $\dim_H$ refers to the Hausdorff dimension, with $\alpha = \dim_H(A)$.
    
%Our results give improvements of previous results due to Orponen and Bourgain and our approach leads to considerably shorter proofs over the previous works.
    
%\end{abstract}
\section{Introduction}
Let $\delta, \sigma\in (0,1)$. A set $A\subset \mathbb{R}$ is called $\delta$-discretized if it is a union of closed intervals of length $\delta$. A $\delta$-discretized set $A$ is called a $(\delta, \sigma)$-set if $|A|\approx \delta^{1-\sigma}$ and it satisfies the non-concentration condition $|A\cap I|\lesssim |I|^\sigma|A|$ for all intervals $I$. Roughly speaking, we can consider a $(\delta, \sigma)$-set as a discrete analogue of a set of Hausdorff dimension $\sigma$.

Bourgain \cite{Bourgain2010} proved that any $(\delta, \sigma)$-set cannot be approximately closed under both addition and multiplication. More precisely, if $A\subset [1, 2]$ is a $(\delta, \sigma)$-set, then there exists $\epsilon=\epsilon(\sigma)>0$ such that 
\begin{equation}\label{SPE}\max \left\lbrace |A+A|, |A\cdot A|\right\rbrace\gtrsim \delta^{-\epsilon}|A|.\end{equation}
Here by $|X|$, we mean the Lebesgue measure of $X$. This result settles a conjecture of Katz and Tao in \cite{KatzTao} for $\sigma=1/2$. The recent work of Guth, Katz, and Zahl \cite{GKZ} provides a new proof of (\ref{SPE}) with explicit exponent $\epsilon$, namely, the estimate (\ref{SPE}) holds for any $0<\epsilon<\frac{\sigma(1-\sigma)}{4(7+3\sigma)}$. 

%Such results have found many applications in the literature on various topics in geometric measure theory and related areas including Borel rings in real line, distance sets, orthogonal and radial projections, Besicovitch and Furstenberg sets, and spectral gaps. We refer the reader to \cite{Bourgain2003, Bourgain2009, Bourgain2010, BouDya, BouGam08, BouGam12, EdMil, ErdVol, Falconerbook, HPS, KatzTao, KatzZahl, Liu, PS1, PS3} and references therein for more details. A number of generalizations with applications in different settings can also be found in \cite{Benoist, Chen, He19, He20, HeSax, Li, RazZahl}.

In this paper, we explore more deeper properties of discretized sum-product type problems with an emphasis on energy variants and applications of the $A+cB$ problem. Let us first start with the following theorem of Bourgain in \cite{Bourgain2010}. Recall that a set $A\subset \R$ is called $\delta$-separated if every two elements in $A$ have distance greater than $\delta$.
\begin{theorem}\label{bourgain} Given $\alpha \in (0,1)$ and $\gamma,\eta > 0$, there exist $\epsilon_{0},\epsilon > 0$, depending only on $\alpha, \gamma, \eta$, such that the following holds for all sufficiently small $\delta > 0$.

Let $\nu$ be a probability measure on $[0,1]$ satisfying $\nu(B(x,r)) \le r^{\gamma}$ for all $x \in \R$ and $\delta < r \le \delta^{\epsilon_{0}}$. Let additionally $A \subset [0,1]$ be a $\delta$-separated set with $|A|\ge \delta^{-\alpha}$, which also satisfies the non-concentration condition $|A \cap B(x,r)| \le r^{\eta}|A|$ for $x \in \R$ and $\delta \le r \le \delta^{\epsilon_{0}}$. 

Then, there exists a point $c \in \spt (\nu)$ such that
\begin{equation*} |A + cA|_{\delta} \ge \delta^{- \epsilon}|A|. \end{equation*}
Here $|\cdot|_{\delta}$ refers to the $\delta$-covering number of $A$, namely the size of the smallest covering of $A$ by intervals of length $\delta$. \end{theorem}

In the above theorem, $c\in \spt(\nu)$. If $C$ is a $\delta$-separated set satisfying $|C \cap B(x,r)| \leq r^{\gamma}|C|$ for all $x \in \R$ and $\delta \leq r \leq \delta^{\epsilon_{0}}$, we choose the uniformly distributed probability measure $\nu$ on the $\delta$-neighbourhood of $C$ such that $\nu(B(x,r)) \lesssim r^{\gamma}$, namely, $\nu(X)=\frac{|X\cap C(\delta)|}{|C|_{\delta}}.$
The theorem above tells us that there exists $c\in \spt(\nu)$ such that 
\[|A+cA|_\delta\ge \delta^{-\epsilon}|A|.\]
This implies that there exists $c\in C$ with $|A + cA|_{\delta} \geq \delta^{- \epsilon}|A|$. This can be explained as follows. Assume $c_\nu$ is such an element in $\spt(\nu)$ and $c\in C$ is an element such that $|c-c_\nu|\le \delta$. We observe that if $\mathcal{C}$ is a covering of $A+cA$ by $\delta$-balls, then for each ball in $\mathcal{C}$, adding two translations to the left and to the right by $\delta$, we would have a covering of $A+c_\nu A$. This gives the desired conclusion. 

Orponen \cite{O21} recently obtained a stronger result that extends Bourgain's result for different sets.
\begin{theorem}\label{thmO1} Let $0 < \beta \le \alpha < 1$ and $\eta > 0$. Then, for every $\gamma \in ((\alpha - \beta)/(1 - \beta),1]$, there exist $\epsilon,\epsilon_{0},\delta_{0} \in (0,\tfrac{1}{2}]$, depending only on $\alpha,\beta,\gamma, \eta$, such that the following holds. Let $\delta \in 2^{-\N}$ with $\delta \in (0,\delta_{0}]$, and let $A,B \subset \delta \Z \cap [0,1]$ satisfy the following hypotheses:
\begin{enumerate}
\item[(A)] \label{A} $|A| \le \delta^{-\alpha}$.
\item[(B)] \label{B} $|B| \ge \delta^{-\beta}$, and $B$ satisfies the following Frostman (non-concentration) condition: 
\begin{displaymath} 
|B \cap B(x,r)| \le r^{\eta}|B|, ~~ \forall x\in \R,~ \delta \le r \le \delta^{\epsilon_{0}}. 
\end{displaymath} 
\end{enumerate}
Further, let $\nu$ be a Borel probability measure with $\spt (\nu) \subset [\tfrac{1}{2},1]$, satisfying the Frostman condition $\nu(B(x,r)) \le r^{\gamma}$ for $x \in \R$ and $\delta \le r \le \delta^{\epsilon_{0}}$. Then, there exists a point $c \in \spt (\nu)$ such that 
\begin{displaymath} |A + cB|_{\delta} \ge \delta^{-\epsilon}|A|. \end{displaymath}
\end{theorem}
%\thang{I will add some more setences about $B+B$ and $C$}

Orponen also made the conjecture that the sharp lower bound for $\gamma$ should be $\gamma>\alpha-\beta$. 
\begin{conjecture}\label{mainConjecture} Let $\alpha,\beta,\gamma \in (0,1)$ with $\alpha \ge \beta$ and $\gamma > \alpha - \beta$. Assume that $A,B \subset [0,1]$ and $C\subset [1/2, 1]$ are $\delta$-separated sets with cardinalities $|A| \le \delta^{-\alpha}$, $|B| = \delta^{-\beta}$, and $|C| = \delta^{-\gamma}$. Assume moreover that $|B \cap B(x,r)| \lesssim r^{\beta}|B|$ and $|C \cap B(x,r)| \lesssim r^{\gamma}|C|$ for all $x \in \R$ and $r > 0$. Then, there exist $\epsilon = \epsilon(\alpha,\beta,\gamma) > 0$ and a point $c \in C$ such that $|A + cB|_{\delta} \gtrsim_{\alpha,\beta,\gamma} \delta^{- \epsilon}|A|$. \end{conjecture}

This conjecture is made based on a number of examples in the discrete setting. Let $A, B, C$ be finite sets in $\mathbb{R}$. It is well-known that one can use the Szemer\'{e}di-Trotter theorem \cite{SzeTrot} to show that if $|B||C|\gtrsim |A|$ then there exists $c\in C$ such that $|A+cB|\gtrsim |A|$. For reader's convenience, we reproduce the argument here. For any $c\in C$, let $L_c$ be the set of lines of the form $x=r-cy$ with $r\in A+cB$. Let $L=\cup_{c\in C}L_c$. Then it is clear that $|L|\le |C|\max_{c\in C}|A+cB|$. We observe that $I(A\times B, L)\ge |A||B||C|$. Thus, the Szemer\'{e}di-Trotter incidence theorem gives 
\[|A||B||C|\lesssim |A|^{2/3}|B|^{2/3}|C|^{2/3}\max_{c\in C}|A+cB|^{2/3},\]
which gives 
\[\max_{c\in C}|A+cB|\gtrsim |A|^{1/2}|B|^{1/2}|C|^{1/2}.\]
In other words, if one wishes to have $|A+cB|\gtrsim |A|^{1+\epsilon}$ for some $\epsilon>0$, then the condition $|B||C|\gtrsim |A|^{1+2\epsilon}$ is needed. The following example, taken from \cite{O21}, also tells us that this condition is sharp. For $n\in \mathbb{N}$, define 
\[A_n=\left\lbrace \frac{1}{n^{1/2}}, ~\frac{2}{n^{1/2}}, \ldots, 1 \right\rbrace, ~B_n=\left\lbrace \frac{1}{n^{1/4}}, ~\frac{2}{n^{1/4}}, \ldots, 1 \right\rbrace=C_n.\]
We can check that for every $\epsilon>0$, there exists $n\in \mathbb{N}$ such that $|A_n+C_nB_n|\le n^{\epsilon}|A|$. 
The same happens in the finite field setting. We refer the reader to \cite{OrVen} for more discussions. 

If we assume $A$ and $B$ are Ahlfors-David regular sets in $[0, 1]$, then Conjecture \ref{mainConjecture} is known to be true in \cite{O22a}. In fact, in \cite{O22a}, Orponen proved a much stronger statement as follows. Let $A, B\subset \mathbb{R}$ be closed sets, where $A$ is $\alpha$-regular and $B$ is $\beta$-regular. Then 
\begin{equation}\label{AD-set}
\dim_H\left\lbrace c\in \mathbb{R}\colon \dim_H(A+cB)<\alpha+\frac{\beta(1-\alpha)}{2-\alpha} \right\rbrace=0.\end{equation}
For general sets, in another paper \cite{O22b}, Orponen proved the following result. 
\begin{theorem}\label{thm-O-dimension}
Let $0<\beta\le \alpha<1$ and $\sigma>0$. Then there exists $\epsilon=\epsilon(\alpha, \beta, \sigma)>0$ such that if $A, B\subset \mathbb{R}$ are Borel sets with $\dim_H(A)=\alpha$, $\dim_H(B)=\beta$, then 
\[\dim_H\{c\in \mathbb{R}\colon \dim_H(A+cB)\le \alpha+\epsilon\}\le \frac{\alpha-\beta}{1-\beta}+\sigma.\]
In particular, 
\[\dim_H\left\lbrace c\in \mathbb{R}\colon \dim_H(A+cB)=\alpha \right\rbrace\le \frac{\alpha-\beta}{1-\beta}.\]
\end{theorem}
To prove this theorem, as mentioned in \cite{O22b}, the following upgraded version of Theorem \ref{thmO1} is crucial. 
\begin{theorem}\label{thm1.1}
Let $0<\beta\le \alpha<1$ and $\eta>0$. Then, for every $\gamma\in ((\alpha-\beta)/(1-\beta), 1]$, there exist $\epsilon_0, \epsilon, \delta_0\in (0, 1/2]$, depending only on $\alpha, \beta, \gamma, \eta$, such that the following holds. Let $\delta\in 2^{-\mathbb{N}}$ with $\delta\in (0, \delta_0]$, and let $A, B\subset \delta \mathbb{Z}\cap [0, 1]$ satisfy the following hypothesis: 
\begin{itemize}
    \item $|A|\le \delta^{-\alpha}$
    \item $|B|\ge \delta^{-\beta}$, and $B$ satisfies the following Frostman condition 
    \[|B\cap B(x, r)|\le r^\eta~ |B|, ~~ \forall x\in \R, \delta\le r\le \delta^\epsilon_0.\]
\end{itemize}
Further, let $\nu$ be a Borel probability measure with $\spt(\nu)\subset [0, 1]$, and satisfying the Frostman condition $\nu(B(x, r))\le r^\gamma$ for all $x\in \mathbb{R}$ and $0<r<\delta^{\epsilon_0}$. Then, there exists $c\in \spt(\nu)$ such that the following holds: if $G\subset A\times B$ is any subset with $|G|\ge \delta^\epsilon |A||B|$, then  
\[|\pi_c(G)|_\delta\ge \delta^{-\epsilon}|A|,\] where $\pi_c(a,b)=a+cb.$
\end{theorem}
For the most recent progress on this topic, we refer the reader to \cite{recent} for more details. 

The main purpose of this paper is to study the energy variant of this problem. More precisely, let $\delta\in 2^{-\mathbb{N}}$, and $A, B\subset \delta \mathbb{Z}\cap [0,1]$ and $C\subset [1/2, 1]$ be $\delta$-separated sets. 
Suppose that 
\[\sum_{c\in C} |\{(a_1, a_2, b_1, b_2)\in A^2\times B^2: |(a_1 +cb_1) - (a_2 + cb_2)|\le \delta\}|_\delta = \frac{1}{K}\cdot |A|_\delta^{3/2}|B|_{\delta}^{3/2}|C|.\]
We aim to give an explicit lower bound of $K$ and discuss a number of applications of this energy estimate. 
\subsection{Main result}
Our main result is stated as follows.

\begin{theorem}\label{thm_energy_III}
Let $\alpha,\beta,\gamma,\eta \in (0,1)$ with $\alpha\ge\beta$. There exist $\epsilon_0, \delta_0\in (0, 1/2]$ depending on $\alpha, \beta, \gamma, \eta$ such that the following holds.\\
Let $A, B\subset \delta \mathbb{Z}\cap [0,1]$ and $C\subset [1/2, 1]$ be $\delta$-separated sets with $\delta\in (0,\delta_0]$.
Assume in addition that $|A|\le \delta^{-\alpha}$, $|B| \ge \delta^{-\beta}$, $|C| = \delta^{-\gamma}$, and 
\begin{align*}
    &|B \cap B(x,r)| \lesssim r^{\eta}|B|,~
|C \cap B(x,r)| \lesssim r^{\gamma}|C|, \quad~\forall x\in \mathbb{R}, ~\delta\le r\le \delta^{\epsilon_0}.
\end{align*}
Suppose that 
\begin{equation*}\label{eq_KABC}
    \sum_{c\in C} |\{(a_1, a_2, b_1, b_2)\in A^2\times B^2: |(a_1 +cb_1) - (a_2 + cb_2)|\le \delta\}|_\delta = \frac{1}{K}\cdot |A|_{\delta}^{3/2}|B|_{\delta}^{3/2}|C|.
\end{equation*}
%\begin{itemize}
 %   \item[(i)] There exist subsets $B'\subset B$, $ C'\subset C$, a point $c^*\in C$, and $\frac{1}{K}\le \rho<1$ such that 
%\[|c^*B'+cB'|_\delta\lesssim (\rho K)^{4m_2+6m_1}|B'|\,, \quad\forall c\in C',\]
%\[|B'\pm B'|_\delta\lesssim (\rho K)^{2m_2+2m_1}|B'|,\]
%\[|B'|\gtrsim \frac{|B|}{(\rho K)^{m_1}}\,,\quad |C'|\gtrsim \frac{|C|\rho|\log\delta|^{-1}}{(\rho K)^{4m_1}}.\]
%    \item[(ii)] 
Then, for all $\epsilon>0$, $K$ can be bounded from below by 
\begin{align*}
    K\gtrsim |\log\delta|^{O(1)}\delta^{O(\epsilon)} \min \{M_0, M_1, M_2, M_3, M_4\},
\end{align*}
where
\[M_0=\delta^{-\frac{\gamma-\beta}{4}},~~\eqref{m0}\]
\[M_1=\delta^{-\frac{\gamma+2\beta+\epsilon_0\eta}{4}}, ~~ \eqref{eq_K_nonempty_R}\]
\[M_2=\max\left\lbrace \delta^{-\frac{3\gamma-8\beta+1-\epsilon_0(3-\gamma)}{16m_2+36m_1}}\,, \delta^{-\frac{4\gamma-9\beta+1-\epsilon_0(3-\eta)}{16m_2+36m_1}} \right\rbrace,~~(\ref{eq:16}, \ref{eq:17})\]
\[M_3=\max\left\lbrace \delta^{-\frac{4\gamma-10\beta-\epsilon_0(3-\gamma)}{16m_2+44m_1}}\,, \delta^{-\frac{5\gamma-11\beta-\epsilon_0(2-\eta)}{16m_2+44m_1}}\,, \delta^{-\frac{5\gamma-12\beta+\eta-\epsilon_0(2-\gamma)}{16m_2+44m_1}}\,, \delta^{-\frac{6\gamma-13\beta+\eta-\epsilon_0(1-\eta)}{16m_2+44m_1}} \right\rbrace,~~(\ref{eq:18}, \ref{eq:19}, \ref{eq:20}, \ref{eq:21})\]
\[M_4=\max\left\lbrace \delta^{-\frac{6\gamma-14\beta-\epsilon_0(3-\gamma)}{20m_2+54m_1}}\,, \delta^{-\frac{7\gamma-16\beta+\eta-\epsilon_0(2-\gamma)}{20m_2+54m_1}}\,, \delta^{-\frac{7\gamma-15\beta-\epsilon_0(2-\eta)}{20m_2+54m_1}}\,, \delta^{-\frac{8\gamma-17\beta+\eta -\epsilon_0(1-\eta)}{20m_2+54m_1}}\right\rbrace, ~~ (\ref{eq:27}, \ref{eq:28}, \ref{eq:29}, \ref{eq:30})\]
for positive constants $m_1, m_2\geq 1$ (given by the Balog-Szmer\'{e}di-Gower theorem~\ref{thm2.3} below).
%\end{itemize}
\end{theorem}
\subsection*{Sketch of main ideas:}
We briefly discuss methods/techniques we use to prove this theorem. We recall that the main purpose is to find lower bounds of $K$ in the following identity:
\[\sum_{c\in C} |\{(a_1, a_2, b_1, b_2)\in A^2\times B^2: |(a_1 +cb_1) - (a_2 + cb_2)|\le \delta\}|_\delta = \frac{1}{K}\cdot |A|_\delta^{3/2}|B|_{\delta}^{3/2}|C|.\]
The first step is to prove the existence of ``good" subsets $B'\subset B$, $ C'\subset C$, point $c^*\in C$, and constant $ \rho\in \big[\frac{1}{K},1\big)$ such that 
\[|c^*B'+cB'|_\delta\lesssim (\rho K)^{4m_2+6m_1}|B'|, ~~\forall c\in C',\]
\[|B'\pm B'|_\delta\lesssim (\rho K)^{2m_2+2m_1}|B'|,\]
\[|B'|\gtrsim \frac{|B|}{(\rho K)^{m_1}}, \quad|C'|\gtrsim \frac{|C|\rho|\log\delta|^{-1}}{(\rho K)^{4m_1}}.\]
To prove the energy theorem, we use Lemma \ref{lem:BourgainIntersection} in the next section to give a lower bound for the set $D=(B'-b_1)\cap (b_2-b_3)C'$, for some $b_1, b_2, b_3\in B'$. More precisely, we have 
\[|D|\gtrsim \frac{|C||\log\delta|^{-1}}{K^{4m_1}|B|}.\]
Note that we will need $|C|\ge K^{4m_1}|B|$ to guarantee that $D$ is non-empty. The proof then proceeds by establishing upper and lower bounds on sum sets of the form $|d_1\widetilde{D}+d_2\widetilde{D}+\cdots+d_2\widetilde{D}|_{\delta},$ for elements $d_i \in C^{'}-C^{'}$ and the set $\widetilde{D}$ is an appropriately large subset of $D$ satisfying small sum set condition. To prove an upper bound, we rely on the $\delta$-covering variant of Pl\"unnecke's inequality, Lemma \ref{lem3.2}, and the estimates of Theorem \ref{thm:struc}, making use of the fact that $D\subset B-b_1$. To obtain a lower bound, we follow the Guth-Katz-Zahl approach \cite{GKZ}, which is the discretized analog of Garaev's method \cite{Garaev07} over finite fields, with appropriate changes along the way clarifying and optimizing many of the steps. One of the biggest challenges in this approach is to optimize all parameters, roughly speaking, based on the definition of $D$, in the proof, if we want to bound the $d$-covering of $D$, for some $d\ge \delta$, then there are two ways one can proceed: either using the non-concentration condition on $C$ or the non-concentration condition on $B$. Using the condition from only one set might imply an empty range. Overall, in the end, we have at least $64$ ranges for $K$. This requires much work to figure out the best range for our purposes. It is very natural to ask if the argument presented in this paper can be improved to solve Conjecture \ref{mainConjecture} completely. At least to us, it is not possible when working with the set $D$. This can be seen clearly from the fact that we need  $|C|\ge K^{4m_1}|B|$ to guarantee $D\ne \emptyset$. So the range $\gamma>\alpha-\beta$ is not sufficient for this purpose. The reader should keep in mind that the lower bounds of $K$ only depend on the sets $B$ and $C$, and are independent of the size or structural properties of $A$. Therefore, this approach has potential for further applications on other sum-product type questions. One might ask if the other sum-product type results over finite fields in \cite{BHS, HH, KMS, MPRRS, MPSS} can be formulated and studied in this setting. We hope to address this question in a sequel paper. 
\subsection{Applications on the $A+cB$ problem}
In the spirit of Theorem \ref{thm1.1}, we first present two applications with explicit values of $\epsilon$.

\begin{theorem}\label{thm-G1l}
 Let $\alpha, \beta, \eta \in (0,1)$, $\beta\le \alpha\le (21\beta+1)/22$. Then, for every $\gamma\in ((78\alpha-66\beta)/6, 1]$, there exist $\epsilon_0, \delta_0\in (0, 1/2]$, depending only on $\alpha, \beta, \gamma, \eta$, such that the following holds. Let $\delta\in 2^{-\mathbb{N}}$ with $\delta\in (0, \delta_0]$, and let $A, B\subset [0, 1]$ be $\delta$-separated sets satisfying 
\begin{itemize}
    \item $|A|\le  \delta^{-\alpha}$
    \item $|B|\ge \delta^{-\beta}$, and $B$ satisfies the following Frostman condition 
    \[|B\cap B(x, r)|\le r^\eta ~ |B|, ~~ \forall x\in \R, \delta\le r\le \delta^{\epsilon_0}.\]
\end{itemize}
Further, let $C$ be a $\delta$-separated set in $[1/2, 1]$ with $|C\cap B(x, r)|\lesssim r^\gamma~ |C|$ for all $x\in \mathbb{R}$ and $0<r<\delta^{\epsilon_0}$. Then, there exists $c\in C$ such that the following holds: if $G\subset A\times B$ is any subset with $|G|\ge \delta^\epsilon |A||B|$, then  $|\pi_c(G)|_\delta\ge \delta^{-\epsilon}|A|.$ Here $\epsilon$ can be chosen arbitrary in the range
\[0<\epsilon<\min \left\lbrace \frac{4\gamma-74\alpha+65\beta+1}{444}, ~\frac{6\gamma-78\alpha+66\beta}{468} \right\rbrace.\]
%for any $\epsilon$ satisfying
%\[0<\epsilon<\min \left\lbrace \frac{4\gamma-74\alpha+65\beta+1}{444}, ~\frac{6\gamma-78\alpha+66\beta}{468} \right\rbrace.\]
%If $G\subset A\times B$ is any subset with $|G|\ge \delta^\epsilon |A||B|$, then  
%\[|\pi_c(G)|_\delta\ge \delta^{-\epsilon}|A|.\]
\end{theorem}
\begin{theorem}\label{thm-G2}
 Let $\alpha, \beta, \eta \in (0,1)$, $\beta\le \alpha$ and $\alpha>(21\beta+1)/22$. Then, for every $\gamma\in ((74\alpha-65\beta-1)/4, 1]$, there exist $\epsilon_0, \delta_0\in (0, 1/2]$, depending only on $\alpha, \beta, \gamma, \eta$, such that the following holds. Let $\delta\in 2^{-\mathbb{N}}$ with $\delta\in (0, \delta_0]$, and let $A, B\subset [0, 1]$ be $\delta$-separated sets satisfying the following hypotheses: 
\begin{itemize}
    \item $|A|\le  \delta^{-\alpha}$
    \item $|B|\ge \delta^{-\beta}$, and $B$ satisfies the following Frostman condition 
    \[|B\cap B(x, r)|\le r^\eta ~ |B|, ~~ \forall x\in \R, \delta\le r\le \delta^{\epsilon_0}.\]
\end{itemize}
Further, let $C$ be a $\delta$-separated set in $[1/2, 1]$ with $|C\cap B(x, r)|\lesssim r^\gamma~ |C|$ for all $x\in \mathbb{R}$ and $0<r<\delta^{\epsilon_0}$. Then, there exists $c\in C$ such that the following holds: if $G\subset A\times B$ is any subset with $|G|\ge \delta^\epsilon |A||B|$, then  $|\pi_c(G)|_\delta\ge \delta^{-\epsilon}|A|.$ Here $\epsilon$ can be chosen arbitrary in the range
\[0<\epsilon<\min \left\lbrace \frac{4\gamma-74\alpha+65\beta+1}{444}, ~\frac{6\gamma-78\alpha+66\beta}{468} \right\rbrace.\]
%If $G\subset A\times B$ is any subset with $|G|\ge \delta^\epsilon |A||B|$, then  
%\[|\pi_c(G)|_\delta\ge \delta^{-\epsilon}|A|.\]
\end{theorem}
While Theorem \ref{thm-G1l} and Theorem \ref{thm-G2} offer explicit exponents for $\epsilon$, the ranges for $\gamma$ is worse than that of Theorem \ref{thm1.1} when $(\alpha-\beta)\to 0$. There is one point we should emphasise here, that if we want $\gamma\to 0$ as $(\alpha-\beta)\to 0$ then we would need $K\ge \delta^{-\epsilon}$, for some $\epsilon>0$, whenever $\gamma>0$. 
Unfortunately, the statement of the energy theorem says that $\gamma$ is bounded from below by a function in $\beta$.

%A quick explanation for this matter will be provided in the section ``sketch of main ideas". 

%On the other hand, since Theorem \ref{thm1.1} was proved by using a sophisticated argument with a number of reductions to Theorem \ref{thmO1}, one might think that the same framework holds for the two theorems above. This is true, but the explicit value of $\epsilon$ would be much worse compared to those presented above. 

\subsection{An energy theorem via incidence bounds}
%below relies on a geometric argument making use of a point-tube incidence bound of Dabrowski, Orponen, and Villa in \cite{DOV}.
In this section, we present an energy theorem, which will be proved by using a point-tube incidence bound of Dabrowski, Orponen, and Villa in \cite{DOV}.

Compared to the first energy theorem, this theorem requires $\delta\le r\le 1$ instead of $\delta\le r\le \delta^{\epsilon_0}$ for some $0<\epsilon_0<1$ and no non-concentration condition on $C$ is needed. 

Recall that given $\delta$-separated sets $A,B,C$, we define $K$ by
\begin{align}\label{defKABC}
    \sum_{c\in C} |\{(a_1, a_2, b_1, b_2)\in A^2\times B^2: |(a_1 +cb_1) - (a_2 + cb_2)|\le \delta\}|_\delta = \frac{1}{K}\cdot |A|_\delta^{3/2}|B|_{\delta}^{3/2}|C|.
\end{align}

\begin{theorem}\label{thm_energy_II}
Let $\alpha,\beta,\gamma \in (0,1)$ satisfying $\alpha+\beta>1$. Let $\delta\in 2^{-\mathbb{N}}$, and let $A, B\subset \delta \mathbb{Z}\cap [0,1]$, $C\subset [1/2, 1]$ be $\delta$-separated sets satisfying $|A|=\delta^{-\alpha}, |B|=\delta^{-\beta}$, $|C|=\delta^{-\gamma}$, and 
\[|A\cap B(x, r)|\le M  r^\alpha|A|,~~|B\cap B(x, r)|\le M  r^\beta|B|,~\forall ~x\in \mathbb{R},~\delta\le r\le 1,\]
for some $M>1$.
%\begin{itemize}
%    \item[(i)] There exist subsets $B'\subset B$, $ C'\subset C$, points $c^*\in C$, and $\frac{1}{K}\le \rho<1$ such that 
%\[|c^*B'+cB'|_\delta\lesssim (\rho K)^{4m_2+6m_1}|B'|, ~~\forall c\in C',\]
%\[|B'\pm B'|_\delta\lesssim (\rho K)^{2m_2+2m_1}|B'|,\]
%\[|B'|\gtrsim \frac{|B|}{(\rho K)^{m_1}}, |C'|\gtrsim \frac{|C|\rho|\log\delta|^{-1}}{(\rho K)^{4m_1}}.\]
%    \item[(ii)] 
Then $K$, defined as in \eqref{defKABC}, can be bounded from below by 
\[K\gtrsim \delta^{\frac{\alpha-3\beta-4\gamma+2\gamma(\alpha+\beta)-\alpha^2+\beta^2+2}{2(3-\alpha-\beta)}}.\]
\end{theorem}

In the statement of Theorem \ref{thm_energy_II}, by taking $\alpha=\beta>1/2$ and $\gamma>1/2$, then we have the bound
\begin{equation*}\label{eqn:EB2alpha=beta}
K\gtrsim \delta^{\frac{-\gamma(4-4\alpha)+2-2\alpha}{2(3-2\alpha)}}.
\end{equation*}

%there is a weakness and an advantage in the statement of Theorem \ref{thm_energy_II}. More precisely, on the non-concentration conditions, we require $\delta\le r\le 1$ which is much stronger than $\delta\le r\le \delta^{\epsilon_0}$ for some $0<\epsilon_0<1$. On the other hand, we do not need any non-concentration condition on $C$, which is quite flexible for applications. 

%in this theorem, to get an explicit lower bound for $K$, we need to pay a cost on the non-concentration conditions of the sets $A$ and $B$, namely, $\delta\le r\le 1$ instead of $\delta\le r\le \delta^{\epsilon_0}$ for some $\epsilon_0>0$. It is not clear, at least to us, if the two conditions can be reduced from each other in some sense. We also note that another difference between the two results is that in the first theorem, we require a non-concentration condition on the set $C$.\thang{Need to say about the non-concentration of C here.}

As in the previous section, we immediately obtain a $\delta$-covering result for the $A+cB$ problem.

\begin{theorem}\label{a+cba41}
Let $\alpha,\beta,\gamma \in (0,1)$ satisfying $\alpha+\beta>1$. Let $\delta\in 2^{-\mathbb{N}}$, and let $A, B\subset \delta \mathbb{Z}\cap [0,1]$, $C\subset [0, 1]$ be $\delta$-separated sets satisfying $|A|\le \delta^{-\alpha}$, $|B|\ge \delta^{-\beta}$, $|C|=\delta^{-\gamma}$, and
\[|A\cap B(x, r)|\lesssim r^\alpha|A|,~~|B\cap B(x, r)|\lesssim r^\beta|B|,~\forall ~x\in \mathbb{R},~\delta\le r\le 1.\]
Then there exists $c\in C$ such that 
\[|A+cB|_\delta\gtrsim \delta^{\frac{-6\beta-4\gamma-2\alpha^2+4\alpha+2\beta^2+2+2\gamma(\alpha+\beta)}{2(3-\alpha-\beta)}}|A|.\]
\end{theorem}

\section{Basic lemmas from Additive Combinatorics}
The first lemma is obtained by using the Cauchy-Schwarz inequality and a dyadic pigeon-hole argument.
\begin{lemma}[Lemma 19, \cite{MP}]\label{intersection}For a finite set $T$ and a collection $\{T_s\colon s\in S\}$ of subsets of $T$, i.e. $T_s\subset T$. Then 
\begin{equation*}
 \left(\sum_{s\in S}|T_s|\right)^2\le |T|\sum_{s,s'\in S}|T_s\cap T_{s'}|.
\end{equation*}%
Further, if there exists \(\delta > 0\) such that
\[
 \sum_{s\in S}|T_s| \ge \delta |S||T|,
\]
then there exists a subset \(P\subseteq S\times S\) such that
\begin{enumerate}
\item \( |T_s\cap T_{s'}|\ge \delta^2 |T| / 2 \) for all pairs \((s,s')\)
   in \(P\).
\item \( |P|\ge \delta^2 |S|^2 / 2 \).
\end{enumerate}
\end{lemma}
The next result is known as the Balog-Szemer\'{e}di-Gowers theorem.
\begin{theorem}[Theorem 6.10, \cite{Tao2008}]\label{thm2.3}
Let $K\ge 1$ be a parameter and $A, B$ be bounded subsets of $\mathbb{R}^n$. If 
\[|\{(a_1, a_2, b_1, b_2)\in A^2\times B^2: |(a_1 +b_1) - (a_2 + b_2)|\le \delta\}|_{\delta}\gtrsim \frac{1}{K}\cdot |A|_\delta^{3/2}|B|_{\delta}^{3/2},
\]
then there exist $A'\subset A$, $B'\subset B$, and constants $m_1, m_2\ge 1$ such that 
\[|A'|_{\delta}\gtrsim K^{-m_1}|A|_\delta, |B'|_\delta\gtrsim K^{-m_1}|B|_\delta,\]
and 
\[|A'+B'|_\delta\lesssim K^{m_2}|A|_\delta^{1/2}|B|_\delta^{1/2}.\]
\end{theorem}
We remark that one can take $m_1=1$ and $m_2=7$ as in \cite{Taoblogb}.
%Note that by Cauchy-Schwarz, the condition $|A+B|_\delta\lesssim K|A|_\delta^{1/2}|B|_\delta^{1/2}$ implies the energy
%\[|\{(a_1, a_2, b_1, b_2)\in A^2\times B^2: |(a_1 +b_1 - (a_2 + b_2)|\le \delta\}|_\delta\gtrsim \frac{1}{K}\cdot |A|_\delta^{3/2}|B|_{\delta}^{3/2}.
%\]
For a  set $X$, we call a set $X^{'}$ a $\delta$-\emph{refinement} of $X$ if  $X^{'}\subset X$, and $|X^{'}|_{\delta} \ge |X|_{\delta}/2$. We recall the following $\delta$-covering version of Pl\"{u}nnecke’s
inequality.
\begin{lemma}[Corollary 3.4, \cite{GKZ}]\label{discretizedPlunnecke}
Let $X,Y_1,\ldots,Y_k$ be  subsets of $\R$. Suppose that 
$|X+Y_i|_{\delta}\le K_i|X|_{\delta}$ for each $i=1,\ldots,k$. Then there is a $\delta$-refinement $X^{'}$ of $X$ so that 
$$
|X^{'} + Y_1 + Y_2  + \ldots + Y_k|_{\delta}\lesssim\big(\Pi_{i=1}^k K_i\big)    |X^{'}|_{\delta}.
$$
In particular,
$$
|Y_1 + Y_2  + \ldots + Y_k|_{\delta}\lesssim\big(\Pi_{i=1}^k K_i \big)   |X|_{\delta}.
$$
\end{lemma}
As mentioned and proved in \cite{GKZ}, we often need to replace the $\delta$-covering of a set with a larger scale. The lemma below also plays an important role for this purpose.
\begin{lemma}[Lemma 2.1, \cite{GKZ}]\label{lm21}
Let $X\subset [1, 2]$ be a $\delta$-separated subset and suppose that $|X|=\delta^{-\sigma}$ for some $0<\sigma<1$. Suppose that \[|X+X|_{\delta}\le K|X|.\]
Then for every $\epsilon>0$, there is a subset $X'\subset X$ with 
$|X'|\gtrsim \delta^{\epsilon}|X|$,
such that 
\[|X'+X'|_{t}\lesssim \delta^{-10\epsilon}K|X'|_{t},\]
for all $\delta<t<1$, with the implicit constants depending on $\sigma$ and $\epsilon$.
\end{lemma}
We recall a simple, but useful result proved in Theorem C of Bourgain's paper \cite{Bourgain2009}.
\begin{lemma}
\label{lem:BourgainIntersection}
Let $X, Y$ be finite subsets of an arbitrary ring and let $M = \max_{y\in Y} |X + yX|$. Then there exist elements $x_1, x_2, x_3 \in X$ such that 
$$
|(X - x_1) \cap (x_2 - x_3) Y| \gtrsim \frac{|Y||X|}{M}.
$$
\end{lemma}
We also need Ruzsa's triangle inequality for finite sets and its $\delta$-covering variant from \cite{GKZ}.
\begin{lemma}[Ruzsa triangle inequality]\label{lem_triangle_inq}
 Let $G$ be an Abelian group and let $X,Y,Z \subset G$ be finite subsets. Then
$$|X-Z| \leq \frac{|X-Y|\cdot |Y-Z|}{|Y|},$$
and 
$$|X+Z| \leq \frac{|X+Y|\cdot |Y+Z|}{|Y|}.$$ 
\end{lemma}

\begin{lemma}[Proposition 3.5, \cite{GKZ}]\label{lem_triangle_inq_delta}
Let $X,Y,Z$ be subsets of $\mathbb{R}$. Then 
\begin{align}\label{eq_tri_inq_delta_1}
    |X-Z|_{\delta} \lesssim \frac{|X-Y|_{\delta}\cdot |Y-Z|_{\delta}}{|Y|_{\delta}},
\end{align}
and 
\begin{align}\label{eq_tri_inq_delta_2}
    |X+Z|_{\delta} \lesssim \frac{|X+Y|_{\delta}\cdot |Y+Z|_{\delta}}{|Y|_{\delta}}.
\end{align}
\end{lemma}
\section{A structural theorem: good subsets of $B$ and $C$}
The main theorem in this section is the following. Roughly speaking, it says that if we have the energy equality 
\[\sum_{c\in C} |\{(a_1, a_2, b_1, b_2)\in A^2\times B^2: |(a_1 +cb_1) - (a_2 + cb_2)|\le \delta\}|_\delta = \frac{1}{K}\cdot |A|_\delta^{3/2}|B|_{\delta}^{3/2}|C|,\]
then one can find subsets $B'\subset B$ and $C'\subset C$ such that the sets $B'+B'$ and $B'+cB'$ have small $\delta$-covering for all $c\in C'$. This theorem can be viewed as the discretized version of \cite[Proposition 4]{MP} due to Murphy and Petridis in the finite field setting.
\begin{theorem}\label{thm:struc}
 Let $\delta\in 2^{-\mathbb{N}}$, and $A, B\subset \delta \mathbb{Z}\cap [0,1]$ and $C\subset [1/2, 1]$ be $\delta$-separated sets. 
Suppose that 
\[\sum_{c\in C} |\{(a_1, a_2, b_1, b_2)\in A^2\times B^2: |(a_1 +cb_1) - (a_2 + cb_2)|\le \delta\}|_\delta = \frac{1}{K}\cdot |A|_\delta^{3/2}|B|_{\delta}^{3/2}|C|.\]
There exist subsets $B'\subset B$, $ C'\subset C$, points $c^*\in C$, and $\frac{1}{K}\le \rho<1$ such that 
\[|c^*B'+cB'|_\delta\lesssim (\rho K)^{4m_2+6m_1}|B'|, ~~\forall c\in C',\]
\[|B'\pm B'|_\delta\lesssim (\rho K)^{2m_2+2m_1}|B'|,\]
\[|B'|\gtrsim \frac{|B|}{(\rho K)^{m_1}}, |C'|\gtrsim \frac{|C|\rho|\log\delta|^{-1}}{(\rho K)^{4m_1}}.\]
\end{theorem}
\begin{proof}
For each $c\in C$, set
\[f_{\delta(A, cB)}:=|\{(a_1, a_2, b_1, b_2)\in A^2\times B^2: |(a_1 +cb_1) - (a_2 + cb_2)|\le \delta\}|_\delta.\]
We are given
\begin{equation}\label{eq1111}
\sum_{c\in C}f_\delta(A, cB) \gtrsim \frac{(|A|_\delta|B|_\delta)^{3/2}|C|}K.
\end{equation}
Observe that 
\[f_\delta(A, cB)\le |A|^2|B|, |A||B|^2,\]
so that we have
$f_\delta(A, cB)\le |A|^{3/2}|B|^{3/2}$.
We now look at the equation (\ref{eq1111}). Without loss of generality, we can assume that 
$f_\delta(A, cB)\gtrsim |A|^{3/2}|B|^{3/2}/K$ 
for all $c\in C$. 
By dyadic pigeonholing, we can find an integer $N\in \N$ and a set $C_1\subset C$ such that for all $c\in C_1$, one has 
\[f_\delta(A, cB)\sim 2^N\frac{|A|^{3/2}|B|^{3/2}}{K}=\frac{|A|^{3/2}|B|^{3/2}}{\rho K}, ~~ \text{ where } \rho=2^{-N}.\]
Notice that $1/K\le \rho\le 1$ and the fact that $2^N\lesssim K$ gives $N\lesssim \log (K)$. Thus, 
\[|C_1|\gtrsim |C|\rho (\log K)^{-1}.\]
If $K>\delta^{-1}$, there is nothing to prove. So we may assume that $K\le \delta^{-1}$ which gives $|C_1|\gtrsim |C|\rho|\log \delta|^{-1}$.

For each $c\in C_1$, by Theorem \ref{thm2.3}, we can find subsets $A^c\subset A$ and $B^c\subset B$ such that 
\begin{align}\label{eq_subset_AB}
    |A^c|\gtrsim \frac{|A|}{(\rho K)^{m_1}},~~ |B^c|\gtrsim \frac{|B|}{(\rho K)^{m_1}}, ~~ |A^c+cB^c|_\delta\lesssim (\rho K)^{m_2}(|A|_\delta|B|_\delta)^{1/2}
\end{align}
for some positive constants $m_2> m_1$.
It is clear that we have 
$|A^c\times B^c|\gtrsim \frac{|A||B|}{(\rho K)^{2m_1}}$, for all $c\in C_1$.
Then Lemma \ref{intersection} ensures that there exists a subset $P\subset C_1\times C_1$ such that 
\[|(A^c\times B^c)\cap (A^{c'}\times B^{c'})|\gtrsim \frac{|A||B|}{(\rho K)^{4m_1}},\]
for all pairs $(c, c')\in P$ and $|P|\gtrsim \frac{|C_1|^2}{(\rho K)^{4m_1}}$. This also yields that there exists $c^*\in C_1$ and $C'\subset C_1$ such that $|C'|\gtrsim \frac{|C|\rho|\log\delta|^{-1}}{(\rho K)^{4m_1}}$ and 
\begin{align}\label{eq_intersection_AB}
    |(A'\times B')\cap (A^c\times B^c)|\gtrsim \frac{|A||B|}{(\rho K)^{4m_1}},
\end{align}
for all $c\in C'$, where we write $A'$ for $A^{c^*}$ and $B'$ for $B^{c^*}$.

For $c \in C_1$, applying the triangle inequality for $\delta$-covering \eqref{eq_tri_inq_delta_2} with $X=Z=cB^{(c)}$, $Y=A^{(c)}$, we have
\begin{equation}
    |B^{(c)}+B^{(c)}|_\delta\sim |cB^{(c)}+cB^{(c)}|_\delta\le \frac{|A^{(c)}+cB^{(c)}|_\delta^2}{|A^{(c)}|} \lesssim (\rho K)^{2m_2} \frac{|A||B|}{|A^{(c)}|} \lesssim (\rho K)^{2m_2+m_1}|B|.
\end{equation}
Similarly, for $c\in C_1$, using the triangle inequality \eqref{eq_tri_inq_delta_1} with $X=Z=cB^{(c)}$, $Y=-A^{(c)}$, one obtains
\begin{align*}
    |B^{(c)}-B^{(c)}|_\delta\sim |cB^{(c)}-cB^{(c)}|_\delta\le \frac{|A^{(c)}+cB^{(c)}|_\delta^2}{|A^{(c)}|} \lesssim (\rho K)^{2m_2} \frac{|A||B|}{|A^{(c)}|} \lesssim (\rho K)^{2m_2+m_1}|B|.
\end{align*}
Similarly, for all $c \in C_1$, one can check that
\begin{equation}
  \label{eq:22}
    |A^{(c)}\pm A^{(c)}|_\delta \lesssim  \frac{|A^{(c)}+cB^{(c)}|_\delta^2}{|B^{(c)}|} \lesssim (\rho K)^{2m_2+m_1}|A|.
\end{equation}
In particular, these estimates imply 
\begin{align}\label{eq_upperbound_Ac_Bc}
    |A'\pm A'|_\delta\lesssim (\rho K)^{2m_2+m_1}|A|, ~~ |B'\pm B'|_\delta\lesssim (\rho K)^{2m_2+m_1}|B|.
\end{align}
Now we turn to estimating $|c^*B'\pm cB'|_\delta$. Again, applying triangle inequalities for $\delta$-covering from Lemma \ref{lem_triangle_inq_delta} with suitable sets, one has
\begin{align*}
  |c^* B'\pm c B'|_\delta 
&\le \frac{|c^*B' + c(B^{c} \cap B')|_\delta|cB' + c(B^{c} \cap B')|_\delta}{|(B^{c} \cap B')|_\delta}\\
&\lesssim \frac{|c^*B' + cB^{c}|_\delta|B' + B'|_\delta}{|(B^{c} \cap B')|_\delta}\\
&\lesssim \frac{|A' + c^* B'|_\delta|A^{c} + c B^{c}|_\delta|B' + B'|_\delta}{|(A^{c} \cap A')|_\delta|(B^{c} \cap B')|_\delta}\\
&\lesssim (\rho K)^{4m_2+5m_1}|B|,
\end{align*}
where the last inequality follows from \eqref{eq_intersection_AB}. This gives that
\begin{equation}\label{eqn:cstarBcB}|B'\pm (c^*)^{-1}cB'|_\delta \lesssim (\rho K)^{4m_2+5m_1} |B|,~\forall~c\in C', \end{equation}
and 
$|B'\pm B'|_\delta\lesssim (\rho K)^{2m_2+m_1}|B|$.
\end{proof}
Using Theorem \ref{thm:struc} and Theorem \ref{bourgain}, the following $\epsilon$-energy variant is straightforward.
\begin{theorem}\label{thm_energy_I}
Let $\alpha,\beta, \gamma, \eta \in (0,1)$ with $\alpha\ge\beta$. There exist $\epsilon, \epsilon_0, \delta_0\in (0, 1/2]$, depending only on $\beta, \gamma, \eta$, such that the following holds for $\delta\in (0, \delta_0]$. Let $A, B\subset \delta \mathbb{Z}\cap [0,1]$ and $C\subset [1/2, 1]$ be $\delta$-separated sets with $\delta\in (0,\delta_0]$.
Assume in addition that $|A|=\delta^{-\alpha}$, $|B|=\delta^{-\beta}$, $|C| = \delta^{-\gamma}$, and 
\begin{align*}
    &|B \cap B(x,r)| \lesssim r^{\eta}|B|,~
|C \cap B(x,r)| \lesssim r^{\gamma}|C|, \quad~\forall x\in \mathbb{R}, ~\delta\le r\le \delta^{\epsilon_0}.
\end{align*}
Suppose that 
\begin{equation*}\label{eq_KABC}
    \sum_{c\in C} |\{(a_1, a_2, b_1, b_2)\in A^2\times B^2: |(a_1 +cb_1) - (a_2 + cb_2)|\le \delta\}|_\delta = \frac{1}{K}\cdot |A|_{\delta}^{3/2}|B|_{\delta}^{3/2}|C|.
\end{equation*}
Then $K$ can be bounded from below by $\delta^{-\epsilon}$.
\end{theorem}
\begin{proof}
    Using Theorem \ref{thm:struc}, we know that there exist subsets $B'\subset B$, $ C'\subset C$, an element $c^*\in C$, and $\frac{1}{K}\le \rho<1$ such that 
\[|c^*B'+cB'|_\delta\lesssim (\rho K)^{4m_2+6m_1}|B'|, ~~\forall c\in C',\]
\[|B'\pm B'|_\delta\lesssim (\rho K)^{2m_2+2m_1}|B'|,\]
\[|B'|\gtrsim \frac{|B|}{(\rho K)^{m_1}}, |C'|\gtrsim \frac{|C|\rho|\log\delta|^{-1}}{(\rho K)^{4m_1}}.\]

To find a lower bound for $K$, we apply Theorem \ref{bourgain} with $A:=B'$ and $C:=C'$. 

Applying Theorem \ref{bourgain} with parameters $\beta/2, \gamma/2$, there exist $\overline{\epsilon}, \overline{\epsilon_0}, \overline{\delta_0}>0$, depending only on $\beta/2, \gamma/2$, such that the following holds for all $\delta\in (0, \overline{\delta_0}]$. If 
\[|B'\cap B(x, r)|\lesssim r^{\beta/2}|B'|,\]
and 
\[|C'\cap B(x, r)|\lesssim r^{\gamma/2}|C'|,\]
for $\delta\le r\le \delta^{\overline{\epsilon_0}}$, then there exists $c\in C'$ such that $|B'+cB'|\gtrsim \delta^{-\overline{\epsilon}}|B'|$. 

If $K>\delta^{-\overline{\epsilon_0}\beta/2}$ or $K>\delta^{-\overline{\epsilon_0}\gamma/8}|\log \delta|^{-1/4}$, then we are done. Thus, we can assume that $K\ll \delta^{-\overline{\epsilon_0}\beta/2}$ and $ K \ll \delta^{-\overline{\epsilon_0}\gamma/8}|\log\delta|^{-1/5}$.

Next, we check the non-concentration conditions of $B'$ and $C'$. We first have 
\[|B'\cap B(x, r)|\le r^\beta|B|\le r^\beta K |B'|.\]
 This gives that 
\[|B'\cap B(x, r)|\ll r^\beta K|B'|\le r^{\beta/2}|B'|,\]
for $\delta\le r\le \delta^{\overline{\epsilon_0}}$.

Similarly, 
\[|C'\cap B(x, r)|\ll r^\gamma |C|\le r^\gamma |C'|K^4|\log \delta|\ll r^{\gamma/2}|C'|,\]
for $\delta\le r\le \delta^{\overline{\epsilon_0}}$. Note that we used the fact that $1/K\le \rho\le 1$. 

Thus, as above, there exists $c\in C'$ such that
\[|B'+cB'|_\delta\gtrsim \delta^{-\overline{\epsilon}}|B'|.\]
Notice that we may replace the set $C'$ above by $(c^*)^{-1}C'$. On the other hand, we know from Theorem \ref{thm:struc} (recall we may take $m_2=7$ and $m_1 = 1$) that
\[|c^*B'+cB'|_\delta\lesssim K^{34}|B'|.\]
This means that $K^{34}\gtrsim \delta^{-\overline{\epsilon}}$, which gives $K\gtrsim \delta^{-\overline{\epsilon}/34}$.

Choose $\delta_0=\overline{\delta_0}$, $\epsilon_0=\overline{\epsilon_0}$, and $\epsilon$ with 
\[\delta^{-\epsilon}=\min\left\lbrace \delta^{-\overline{\epsilon}}, ~\delta^{-\overline{\epsilon_0}\beta/2}, ~\delta^{-\overline{\epsilon_0}\gamma/8}\right\rbrace.\]
This completes the proof. 
\end{proof}
\begin{remark}
In the above proof, if we apply Theorem \ref{thmO1} to the set $A^c+cB^c$ from (\ref{eq_subset_AB}), then a lower bound for $K$ is obtained. However, we would need the condition that $\gamma>(\alpha-\beta)/(1-\beta)$. 
\end{remark}

\section{Proof of Theorem \ref{thm_energy_III}}
Using Theorem \ref{thm:struc}, we know that there exist subsets $B'\subset B$, $ C'\subset C$, an element $c^*\in C$, and $\frac{1}{K}\le \rho<1$ such that 
\[|c^*B'+cB'|_\delta\lesssim (\rho K)^{4m_2+6m_1}|B'|, ~~\forall c\in C',\]
\[|B'\pm B'|_\delta\lesssim (\rho K)^{2m_2+2m_1}|B'|,\]
\[|B'|\gtrsim \frac{|B|}{(\rho K)^{m_1}}, |C'|\gtrsim \frac{|C|\rho|\log\delta|^{-1}}{(\rho K)^{4m_1}}.\]
Notice that it is not possible to expect that the set $c^*B'+cB'$ is $\delta$-separated for each $c\in C'$. This means that for each $c$, the size of $c^*B'+cB'$ is not the same as its $\delta$-covering. So to proceed further, we start noting the following fact that 
\begin{equation}\label{eq_boundB_cB}
    |B'+cB'|\le |B'|^2\,,\quad\quad \text{ for all } c\in C.
\end{equation}
We remark here that it is not possible to expect $|B'+cB'|\le |B'|^{2-\epsilon}$ for any $\epsilon>0$ when $|C'|$ is much larger than $|B'|$. Indeed, if $N$ is the number of tuples $(c, b_1, b_2, b_3, b_4)\in C'\times B'\times B'\times B'\times B'$ such that $b_1+cb_2=b_4+cb_3$, by the Cauchy-Schwarz inequality, one has $N$ is at least $|C'||B'|^{2+\epsilon}$. We observe that the equation $b_1+cb_2=b_4+cb_3$
is equivalent to $b_1-b_4=c(b_2-b_3).$
For a fixed $b_4$, the above identity gives an incidence between the line defined by $y=c(x-b_3)+b_4$
and the point $(b_2, b_1)\in B'\times B'$. So by the Szemer\'{e}di-Trotter theorem \cite{SzeTrot} and taking the sum over all $b_4\in B'$, $N$ is at most $|B'|^3|C'|^{2/3}+|B'|^2|C'|,$
which is smaller than $|B'|^{2+\epsilon}|C'|$ whenever $|C'|\ge |B'|^{3-3\epsilon}$.
\\
With the fact \eqref{eq_boundB_cB} in hand, we can apply Lemma \ref{lem:BourgainIntersection} to find $b_1, b_2, b_3\in B'$ satisfying 
\begin{align}\label{eq_card_D}
    \left\vert (B'-b_1)\cap (b_2-b_3)C' \right\vert\ge \frac{|C'|}{|B'|}\gtrsim \frac{|C||\log \delta|^{-1}}{K^{4m_1}|B|}.
\end{align}
Define 
\begin{align*}
    D:=(B'-b_1)\cap (b_2-b_3)C'.
\end{align*}
If $K^{4m_1}\ge \frac{|C|}{|B||\log \delta|}$, i.e. 
\begin{equation}\label{m0}
    K\ge \delta^{-\frac{\gamma-\beta}{4}}|\log\delta|^{1/4},
\end{equation}
then we are done. Otherwise, we can assume that $|C|>K^{4m_1}|B||\log\delta|$. This condition implies that $D$ is non-empty.

Let $\kappa=\epsilon_0\in (0,1/2)$ be parameters that will be specified later. 
To proceed further, we need the following two lemmas.

\begin{lemma}\label{lem3.2}
Let $d_1=c_1-c_2$ and $d_2=c_3-c_4$ with $c_i\in C'$. Assume that $|c_1-c_2|\le |c_3-c_4|$ and $|c_3-c_4|>\delta^\kappa=\delta^{\epsilon_0}$. Then, for any positive integer $k\ge 1$,
\begin{enumerate}
    \item[(i)] With the non-concentration on $C$, one has 
    \[|d_1D+\underbrace{d_2D+\cdots+d_2D}_{k~\mbox{terms}}|_{\delta}\lesssim \frac{|B|^{4k+4}K^{(8m_2+18m_1)(k+1)}|\log \delta|^{2k+2}}{|b_2-b_3||C|^{2k+1}}\cdot \max\{|d_1|, |d_2|\}\delta^{\epsilon_0(\gamma-1)}.\]
    \item[(ii)] With the non-concentration on $B$, one has
    \[|d_1D+\underbrace{d_2D+\cdots+d_2D}_{k~\mbox{terms}}|_{\delta}\lesssim \frac{|B|^{4k+5}K^{(8m_2+18m_1)(k+1)}|\log\delta|^{2k+2}}{|C|^{2k+2}}\cdot \max\{|d_1|, |d_2|\} \delta^{\epsilon_0(\eta-1)}.\]
\end{enumerate}
\end{lemma}
\begin{proof}
First, observe that for $k$ bounded, the set $d_1D+\underbrace{d_2D+\cdots+d_2D}_{k~\mbox{terms}}$ is contained in an interval of length $d$, with $d\sim \max\{|d_1|, |d_2|\}$.

Let  $c^* \in C$ as in Theorem \ref{thm_energy_I}, and let $D'$ be a $\delta$-refinement of $D$. Then we have 
\begin{align}\label{eq_energy:3.2.1}
|d_1D+d_2D+\cdots+d_2D|_{\delta}&\lesssim |c^*D'|_d^{-1}|c^*D'+d_1D+d_2D+\cdots+d_2D|_{\delta} \\ \nonumber &\le \frac{|c^*D'+c_1D - c_2 D+c_3D - c_4D+\cdots+c_3D-c_4D|_{\delta}}{|c^*D'|_d}.\end{align}
Thus, the task is now to estimate two terms $|c^*D'|_d$ and $|c^*D'+c_1D -c_2D+c_3D - c_4D+\cdots+c_3D-c_4D|_{\delta}$. 

We will handle the second term first. By Theorem \ref{thm_energy_I} and the fact that $D'\subset B'-b_1$, one has
\[|c^* D'+cD|_{\delta}\le|c^* B'+cB'|_{\delta}\le (\rho K)^{4m_2+5m_1}|B|.\]
Then we use Lemma \ref{discretizedPlunnecke} to bound the second term, namely, it is at most
\begin{align}\label{eq_est_second term}
\frac{K^{(4m_2+5m_1)(2k+2)}|B|^{2k+2}}{|c^*D'|_{\delta}^{2k+1}} \lesssim \frac{|B|^{4k+3}|\log \delta|^{2k+1}K^{(4m_2+9m_1)(2k+1)+4m_2+5m_1}}{|C|^{2k+1}},
\end{align}
where the latter follows by $|c^*D|_{\delta}\sim |D|_{\delta}\gtrsim |D|\gtrsim \frac{|C||\log\delta|^{-1}}{|B|K^{4m_1}}$.

For the first term, there are two ways to bound it, which will give two different bounds. One way is to use the non-concentration condition on $C$, and the other way is to use the non-concentration condition on $B$.

{\bf (i) First approach: Non-concentration on $C$.}

Recall that $D'\subset (b_2-b_3)C'$ is a $\delta$-refinement of $D$ and $|b_2-b_3|\ge \delta$ by the $\delta$-separated property of $B$. In particular, this yields that 
\[|D'|_d\sim \left|(b_2-b_3)^{-1}D\right|_{d/(|b_2-b_3|)},\]
where $(b_2-b_3)^{-1}D\subset C'$.

Let $I$ be an arbitrary interval of length $d/|b_2-b_3|$. From assumption, we know that 
$\frac{d}{|b_2-b_3|}\ge d>\delta^{\kappa}=\delta^{\epsilon_0}$.
Thus, the non-concentration condition on the set $C$ implies that
\[|C\cap I|\lesssim \frac{d}{|b_2-b_3|}|C|\delta^{\epsilon_0(\gamma-1)}\sim d |b_2-b_3|^{-1} |C|\delta^{\epsilon_0(\gamma-1)}.\]
 Since $c^*$ is greater than $1/2$, using the above estimate and \eqref{eq_card_D}, we will have an estimate for $|c^*D'|_d$, namely
\[|c^*D'|_d\gtrsim \frac{|D|}{d |b_2-b_3|^{-1} |C|\delta^{\epsilon_0(\gamma-1)}}\gtrsim \frac{|\log\delta|^{-1}|b_2-b_3|}{d \delta^{\epsilon_0(\gamma-1)}|B|K^{4m_1}}.\]

Substituting this and \eqref{eq_est_second term} into \eqref{eq_energy:3.2.1}, one has the upper bound 
\[|d_1D+d_2D+\cdots+d_2D|_{\delta}\lesssim \frac{|B|^{4(k+1)}K^{(8m_2+18m_1)(k+1)}|\log \delta|^{2k+2}}{|b_2-b_3||C|^{2k+1}}\cdot \max\{|d_1|, |d_2|\}\delta^{\epsilon_0(\gamma-1)}.\]

{\bf (ii) Second approach: Non-concentration on $B$}

From hypothesis, we know that $|c_3-c_4|\ge \delta^{\epsilon_0}$. By a similar argument, using the non-concentration assumption on $B$ instead, we have
\[|c^*D'|_d\gtrsim \frac{|D|}{d\delta^{\epsilon_0(\eta-1)}|B|}\gtrsim \frac{|C||\log \delta|^{-1}}{|B|^2K^{4m_1}d\delta^{\epsilon_0(\eta-1)}}.\]
Putting this into \eqref{eq_energy:3.2.1}, we obtain the bound
\[\frac{|B|^{4k+5}K^{(8m_2+18m_1)(k+1)}|\log\delta|^{2k+2}}{|C|^{2k+2}}\cdot \max\{|d_1|, |d_2|\} \delta^{\epsilon_0(\eta-1)}.\]
We complete the proof of the lemma.
\end{proof}
We remark that the above lemma is most effective when $k=1$. For $k>1$, we have the following refinement.  
%------------LEMMA 3.3----------------------------------------
\begin{lemma}\label{lem3.3}
Let $d_1=c_1-c_2$ and $d_2=c_3-c_4$ with $c_i\in C'$. Assume that $|c_1-c_2|\le |c_3-c_4|$ and $|c_3-c_4|>\delta^\kappa=\delta^{\epsilon_0}$. Then, for a bounded positive integer $k\ge 1$ and for any $\epsilon>0$, there exists $\widetilde{D}\subset D$ with $|\widetilde{D}|\ge \delta^\epsilon |D|$, such that the following holds.
\begin{enumerate}
    \item[(i)] With the non-concentration on $C$, one has 
    \begin{align*}
        &|d_1\widetilde{D}+\underbrace{d_2\widetilde{D}+\cdots+d_2\widetilde{D}}_{k~\text{terms}}|_{\delta}\\&\lesssim \delta^{-O(\epsilon)}\frac{|B|^{2k+6}K^{(2m_2+5m_1)k+14m_2+31m_1}|\log\delta|^{k+3}}{|C|^{k+2}|b_2-b_3|}\cdot \max\{|d_1|, |d_2|\}\delta^{\epsilon_0(\gamma-1)}.
    \end{align*}
    \item[(ii)] With the non-concentration on $B$, one has
    \begin{align*}
    &|d_1\widetilde{D}+\underbrace{d_2\widetilde{D}+\cdots+d_2\widetilde{D}}_{k~\text{terms}}|_{\delta}\\&\lesssim \delta^{-O(\epsilon)}\frac{|B|^{2k+7}K^{(2m_2+5m_1)k+14m_2+31m_1}|\log\delta|^{k+3}}{|C|^{k+3}}\cdot \max\{|d_1|, |d_2|\}\delta^{\epsilon_0(\eta-1)}.
    \end{align*}
\end{enumerate}
\end{lemma}
\begin{proof}
As in the previous proof, we identify a $\delta$-refinement $D'$ of $D$, which will be needed when applying Lemma \ref{discretizedPlunnecke} later. Then we apply Lemma \ref{lm21} to find a subset $\widetilde{D}\subset D'$ such that $|\widetilde{D}|\ge \delta^{O(\epsilon)}|D|$ for any given $\epsilon>0$, so that
\[|\widetilde{D}+\widetilde{D}|_{\delta/d_2}\lesssim \delta^{-O(\epsilon)}\frac{|D+D|_\delta}{|D|}|d_2\widetilde{D}|_\delta.\]
On the other hand, using Theorem \ref{thm_energy_I} and \eqref{eq_card_D}, we have
\[\frac{|D+D|_\delta}{|D|}\lesssim \frac{ K^{2m_2+m_1}|B|}{|D|}=\frac{K^{2m_2+5m_1}|B|^2|\log\delta|}{|C|}.\]
Thus, when considering the sum of $k$ terms, applying Lemma \ref{discretizedPlunnecke}, we have
\[|\underbrace{d_2\widetilde{D}+\cdots+d_2\widetilde{D}}_{k \text{ terms}}|_\delta
\lesssim\delta^{-O(\epsilon)}\frac{K^{(2m_2+5m_1)(k-1)}|B|^{2(k-1
)}|\log\delta|^{k-1}}{|C|^{k-1}}\cdot |d_2\widetilde{D}|_\delta.\]
Next, applying \eqref{eq_est_second term} from the proof of the previous lemma with $k=1$, one obtains
\[|c^\ast \widetilde{D}+d_1D+d_2D|_\delta\lesssim \frac{|B|^7K^{16m_2+32m_1}|\log\delta|^3}{|C|^3}.\]
Thus, 
\begin{align*}
&|c^*\widetilde{D}+d_1\widetilde{D}+\underbrace{d_2\widetilde{D}+\cdots+d_2\widetilde{D}}_{k \text{ terms}}|_\delta\\&\lesssim \frac{|d_2\widetilde{D}+d_2\widetilde{D}+\cdots+d_2\widetilde{D}|_\delta}{|d_2\widetilde{D}|_\delta}\cdot \frac{|c^*\widetilde{D}+d_1\widetilde{D}+d_2\widetilde{D}|_\delta}{|d_2\widetilde{D}|_\delta}\cdot |d_2\widetilde{D}|_\delta\\
    &\lesssim \delta^{-O(\epsilon)}\frac{|B|^7K^{16m_2+32m_1}|\log\delta|^3}{|C|^3}\cdot \frac{K^{(2m_2+5m_1)(k-1)}|B|^{2(k-1)}|\log\delta|^{k-1}}{|C|^{k-1}}.
\end{align*}
This gives the bound
\begin{equation}\label{eq_second_bound_2}
    \delta^{-O(\epsilon)}\frac{|B|^{2k+5}K^{(2m_2+5m_1)k+14m_2+27m_1}|\log\delta|^{k+2}}{|C|^{k+2}}.
\end{equation}
Now based on the same argument as in the previous proof, we use different non-concentration conditions to bound the following estimate
\[|d_1\widetilde{D}+\underbrace{d_2\widetilde{D}+\cdots+d_2\widetilde{D}}_{k \text{ terms}}|_\delta\lesssim |c^*\widetilde{D}|_\delta^{-1}|c^*\widetilde{D}+d_1\widetilde{D}+\underbrace{d_2\widetilde{D}+\cdots+d_2\widetilde{D}}_{k \text{ terms}}|_\delta.\]

{\bf (i) Non-concentration on $C$:}

Using the non-concentration condition on $C$, one has the first term is bounded below by
\[|c^*\widetilde{D}|_d\gtrsim \frac{|D|}{d |b_2-b_3|^{-1} |C|\delta^{\epsilon_0(\gamma-1)}}\gtrsim \frac{|\log\delta|^{-1}|b_2-b_3|}{|B|K^{4m_1}d \delta^{\epsilon_0(\gamma-1)}}.\]
Combining with the estimate for the second term \eqref{eq_second_bound_2}, we get
\[\delta^{-O(\epsilon)}\frac{|B|^{2k+6}K^{(2m_2+5m_1)k+14m_2+31m_1}|\log\delta|^{k+3}}{|C|^{k+2}|b_2-b_3|}\cdot \max\{|d_1|, |d_2|\}\delta^{\epsilon_0(\gamma-1)}.\]

{\bf (ii) Non-concentration on $B$:}

In this case, the only difference is that 
\[|c^*\widetilde{D}|_d\gtrsim \frac{|D|}{d\delta^{\epsilon_0(\eta-1)} |B|}\gtrsim \frac{|C||\log \delta|^{-1}}{|B|^2K^{4m_1}d\delta^{\epsilon_0(\eta-1)}}.\]
This follows by the assumption on $B$. Therefore, we obtain the bound 
\[\delta^{-O(\epsilon)}\frac{|B|^{2k+7}K^{(2m_2+5m_1)k+14m_2+31m_1}|\log\delta|^{k+3}}{|C|^{k+3}}\cdot \max\{|d_1|, |d_2|\}\delta^{\epsilon_0(\eta-1)}.\]
This completes the proof of the lemma.
\end{proof}

We now consider the set
$$ 
R:=R(\widetilde{D}) = \left\{ \frac{d_1 - d_2}{d_3 - d_4} : d_i \in \widetilde{D}, |d_3 - d_4| > \delta^\kappa \right\}. 
$$
Since $R(\widetilde{D})\subseteq R(C^{'})$ and $C^{'}\subset [1/2,1]$, we have that $R\subset [-\delta^{-\kappa}/2,\delta^{-\kappa}/2].$

\begin{lemma}\label{lem_nonempty_R}
Suppose $\widetilde{D}$ is non-empty. We have either $R$ is non-empty or
\begin{align}\label{eq_K_nonempty_R}
    K\gtrsim |\log \delta|^{O(1)}\delta^{-\frac{\gamma-2\beta+\kappa \eta}{4}}.
\end{align}
\end{lemma}
\begin{proof}
To prove that $R$ is non-empty, it is enough to show that there are at least two elements $d_3, d_4\in \widetilde{D}$ such that $|d_3-d_4|>\delta^{\kappa}=\delta^{\epsilon_0}$. Since the set $\widetilde{D}$ is assumed to be non-empty, there exists $d_3\in \widetilde{D}$. If $|\widetilde{D}|\gtrsim \delta^{\epsilon_0\eta}|B|$, then the existence of $d_4$ follows from the non-concentration of $B$. Otherwise, from the lower bound on the size of $\widetilde{D}$ in \eqref{eq_card_D}, we have 
\[K^{4m_1}\gtrsim \frac{|C||\log\delta|^{-1}}{|B|^2\delta^{\epsilon_0\eta}}.\]
From this, we will get the lower bound for $K$ and complete the proof of the lemma.
\end{proof}

 Note that we also have $0,1\in R$. Choose a positive integer $m$ so that $2^{-m}\sim\delta^{1-2\kappa}|b_2-b_3|$. Define $s= 2^{-m}$.

The following result was proved in \cite{GKZ}, and it turns out to be useful for our proofs.
\begin{lemma}\label{denseVsGapLemma}
At least one of the following two statements must hold.
\begin{itemize}
\item[(A):] There exists a point $r\in R\cap [0,1]$ with
$$
\max\Big\{\operatorname{dist}(r/2,\ R),\ \ \operatorname{dist}\Big(\frac{r+1}{2},\ R\Big)\Big\}\ge s.
$$
\item[(B):] $|R\cap [0,1]|_s \gtrsim s^{-1}.$
\end{itemize}
\end{lemma}
As in \cite{GKZ}, we refer to the case $(A)$ as the ``gap" case, and $(B)$ as the ``dense" case. We proceed to obtain lower bounds for $K$ corresponding to each case $(A)$ and $(B)$.

\subsection{Dense case}
Suppose we are in the dense case $(B)$, that is $|R\cap [0,1]|_s\gtrsim s^{-1}$. By pigeonholing, we can select points $c_1, c_2, c_3, c_4 \in C^{'}$ with $|c_3 - c_4||b_2-b_3| > \delta^\kappa$ and $|c_1-c_2|\le |c_3-c_4|$ so that
\begin{equation} \label{closebsparse} \left|\bigg\{ (d_1, ..., d_4) \in {\widetilde{D}}^4\colon \Big| \frac{d_1 - d_2}{d_3 - d_4} - \frac{c_1 - c_2}{c_3 - c_4}\Big| < \delta^{1 - 2\kappa}|b_2-b_3|,\  |d_3 - d_4| > \delta^{\kappa} \bigg\}\right|  \lesssim 
|\widetilde{D}|^4 \delta^{1-2\kappa}|b_2-b_3|.
\end{equation}
Here and below, we write $d_1'-d_2'=(b_2-b_3)(c_1-c_2)$ and $d_3'-d_4'=(b_2-b_3)(c_3-c_4)$.

Observe that $|c_3-c_4|>\delta^{\kappa}=\delta^{\epsilon_0}$. Applying Lemma \ref{lem3.3} with $k=1$, the following quantity
\begin{align*}
   \big|(d_1' - d_2') \widetilde{D} + (d_3' - d_4') \widetilde{D}\big|_{\delta|b_2-b_3|} =\left|(c_1-c_2)\widetilde{D}+(c_3-c_4)\widetilde{D}\right|_\delta
\end{align*}
is bounded from above by
\begin{equation}\label{12}
  \delta^{-O(\epsilon)}\frac{|B|^8K^{16m_2+36m_1}|\log\delta|^4}{|C|^3|b_2-b_3|}\cdot |c_3-c_4|\delta^{\epsilon_0(\gamma-1)}
\end{equation}
if we use the non-concentration condition on $C$, and by
\begin{equation}\label{13}
  \delta^{-O(\epsilon)}\frac{|B|^9 K^{16m_2+36m_1}|\log\delta|^4}{|C|^4}\cdot |c_3-c_4|\delta^{\epsilon_0(\eta-1)},
\end{equation}
if we use the non-concentration condition on $B$ instead.

We will now establish a lower bound for $\big|(d_1' - d_2') \widetilde{D} + (d_3' - d_4') \widetilde{D}\big|_{\delta|b_2-b_3|}$. Define $Q \subset {\widetilde{D}}^4$ to be the set of quadruples obeying
\begin{equation} \label{quad} (d_3' - d_4') d_1 + (d_1' - d_2') d_4 = (d_3' - d_4') d_2 + (d_1' - d_2') d_3 + O(\delta|b_2-b_3|). \end{equation}
By the Cauchy-Schwarz inequality, one has $ \big|  (d_1' - d_2') \widetilde{D}+(d_3' - d_4')\widetilde{D}\big|_\delta \gtrsim  |\widetilde{D}|^4 / |Q|$. Thus our goal is now to find an upper bound for $|Q|$.  Note that equation \eqref{quad} can be written as
$$ d_1 + \frac{c_1 - c_2}{c_3 - c_4} d_4 = d_2 + \frac{c_1 - c_2}{c_3 - c_4} d_3 + O(\delta |d_3' - d_4'|^{-1}|b_2-b_3|), $$
which implies that
\begin{equation}\label{smallDifference} 
\left| \frac{d_1 - d_2}{d_3 - d_4} - \frac{c_1 - c_2}{c_3 - c_4} \right| \lesssim \delta |d_3' - d_4'|^{-1} |d_3 - d_4|^{-1}|b_2-b_3|.
\end{equation}
At this step, we consider two separate cases: 

{(i)}: At least $\frac{|Q|}{2}$ quadruples $(d_1, ..., d_4) \in Q$ satisfy $|d_3 - d_4| > \delta^{\kappa}$. For each such quadruple, from inequality \eqref{smallDifference} above, one has
$$ \left| \frac{d_1 - d_2}{d_3 - d_4} - \frac{c_1 - c_2}{c_3 - c_4} \right| \lesssim \delta^{1 - 2 \kappa}|b_2-b_3|. $$
Comparing with \eqref{closebsparse}, we see that the number of such quadruples is $\lesssim |\widetilde{D}|^4 \delta^{1- 2 \kappa}|b_2-b_3|$. Thus if at least half the quadruples $(d_1, ..., d_4)\in Q$ satisfy $|d_3 - d_4| > \delta^{\kappa}$, then 
\begin{equation*}
\big| (c_1 - c_2)\widetilde{D} + (c_3-c_4)\widetilde{D}\big|_\delta \gtrsim \frac{|\widetilde{D}|^4}{|Q|}\gtrsim\delta^{2\kappa-1}|b_2-b_3|^{-1}.
\end{equation*}
Combining with the upper bound in \eqref{12}, we have
\begin{equation*}
\delta^{-O(\epsilon)}\frac{|B|^8K^{16m_2+36m_1}|\log\delta|^4}{|b_2-b_3||C|^3}\cdot |c_3-c_4|\delta^{\epsilon_0(\gamma-1)}\gtrsim \frac{\delta^{2\kappa-1}}{|b_2-b_3|}.
\end{equation*}
This infers that
\[K^{16m_2+36m_1}\gtrsim |\log\delta|^{O(1)}\delta^{O(\epsilon)}\frac{|C|^3}{\delta^{1-2\kappa}|c_3-c_4|\delta^{\epsilon_0(\gamma-1)}|B|^8}.
\]
In other words, one obtains a lower bound for $K$, namely
\begin{equation}\label{eq:16}
K\gtrsim |\log\delta|^{O(1)}\delta^{O(\epsilon)}\delta^{-\frac{3\gamma-8\beta+1-\kappa(3-\gamma)}{16m_2+36m_1}}.
\end{equation}
Here we use the fact that $1>|c_3-c_4|>\delta^\kappa$.
Similarly, if we use the bound (\ref{13}) in place of (\ref{12}), we have 
\begin{equation}\label{eq:17}
K\gtrsim |\log\delta|^{O(1)}\delta^{O(\epsilon)}\delta^{-\frac{4\gamma-9\beta+1-\kappa(3-\eta)}{16m_2+36m_1}}.
\end{equation}

{(ii)} More than half of the quadruples $(d_1, ..., d_4) \in Q$ satisfy $|d_3 - d_4| \le \delta^{\kappa}$. We will make use of non-concentration assumptions on $C$ and $B$ again to estimate the upper bound for the number of these quadruples. As a result, we will obtain the corresponding lower bounds for $K$.

{(a) Non-concentration on $C$:} We begin by choosing elements $d_1,d_4\in \widetilde{D}$. 
Since $\widetilde{D}\subset (b_2-b_3)C$, according to the Frostman condition on $C$ and the requirement that $|d_3-d_4|\le\delta^{\kappa}$, the number of admissible $d_3$ is at most 
\begin{align*}
    |\widetilde{D}\cap B(d_4, \delta^\kappa)| \leq \frac{|C|\delta^{\kappa}}{|b_2-b_3|}\delta^{\epsilon_0(\gamma-1)}.
\end{align*}
Here we used the fact that 
$\frac{\delta^\kappa}{|b_2-b_3|}\ge\delta^\kappa=\delta^{\epsilon_0}$.

Next, observe that from \eqref{smallDifference}, $d_2$ must lie in an interval of length at most $\delta|c_3-c_4|^{-1}$. 
Notice that since $\epsilon_0+\kappa<1$, we must have
\[\frac{\delta}{|c_3-c_4||b_2-b_3|}=\frac{\delta}{|d_3'-d_4'|}<\delta^{1-\kappa}<\delta^{\epsilon_0}.\]
The Frostman condition on $C$ yields that the number of admissible $d_2$ is bounded by 
\begin{align*}
    \delta^\gamma|C||c_3-c_4|^{-\gamma}|b_2-b_3|^{-\gamma}=|d_3'-d_4'|^{-\gamma}\lesssim \delta^{-\kappa\gamma}.
\end{align*}
Thus the set of quadruples of this type has size at most
$\frac{|\widetilde{D}|^2|C|\delta^\kappa\delta^{\epsilon_0(\gamma-1)}}{|b_2-b_3|\delta^{\kappa\gamma}}$.
From \eqref{eq_card_D}, one has
\[|(c_1-c_2)\widetilde{D}+(c_3-c_4)\widetilde{D}|_\delta \gtrsim\frac{|\widetilde{D}|^4}{|Q|}\gtrsim \frac{|\widetilde{D}|^2|b_2-b_3|}{|C|\delta^{\kappa+\epsilon_0(\gamma-1)}\delta^{-\kappa\gamma}}\gtrsim\delta^{O(\epsilon)} \frac{|C||\log\delta|^{-2}|b_2-b_3|}{K^{8m_1}|B|^2}.\]
Altogether, the lower and the upper bounds from (\ref{12}) imply that
\[K^{16m_2+44m_1}\gtrsim |\log\delta|^{O(1)}\delta^{O(\epsilon)}\frac{|C|^4}{|B|^{10}}\frac{|b_2-b_3|^2}{|c_3-c_4|\delta^{\epsilon_0(\gamma-1)}}.\]
Using the fact that $|b_2-b_3||c_3-c_4|>\delta^\kappa$, one gets
\begin{equation}\label{eq:18}
K\gtrsim |\log\delta|^{O(1)}\delta^{O(\epsilon)}\delta^{-\frac{4\gamma-10\beta-\kappa(3-\gamma)}{16m_2+44m_1}}.
\end{equation}
On the other hand, if we use (\ref{13}) instead of (\ref{12}), we have
\begin{equation}\label{eq:19}
K\gtrsim |\log\delta|^{O(1)}\delta^{O(\epsilon)}\delta^{-\frac{5\gamma-11\beta-\kappa(2-\eta)}{16m_2+44m_1}}.
\end{equation}

\text{(b) Non-concentration on $B$:}\\
Similarly, for given elements $d_4\in \widetilde{D}$, the Frostman condition on $B$ implies that the number of admissible $d_3$ is at most 
$|B|\delta^{\kappa\eta}$. \\
Next, observe that from inequality \eqref{smallDifference}, we must have $d_2$ lie in an interval of length at most $\delta|c_3-c_4|^{-1}$, for any fixed $d_1\in \widetilde{D}$.
Because $\epsilon_0+\kappa<1$, one has
$\delta|c_3-c_4|^{-1}< \delta^{1-\kappa}<\delta^{\epsilon_0}$.
Hence the number of admissible $d_2$ is at most $\delta^{\eta}|B||c_3-c_4|^{-\eta}$, where we make use of the Frostman condition on $B$. \\
Altogether, the set of quadruples of this type has size at most
$\frac{|\widetilde{D}|^2|B|^2\delta^{\kappa\eta+\eta}}{|c_3-c_4|^\eta}$, which implies
\[|(c_1-c_2)\widetilde{D}+(c_3-c_4)\widetilde{D}|_\delta
\gtrsim\frac{|\widetilde{D}|^4}{|Q|}
\gtrsim \frac{|\widetilde{D}|^2|c_3-c_4|^\eta}{|B|^2\delta^{\kappa\eta+\eta}}
\gtrsim 
\delta^{O(\epsilon)}\frac{|C|^2|\log\delta|^{-2}}{|B|^4K^{8m_1}\delta^{\eta}}.\]
Then combining this estimate with the upper bound (\ref{12}), one has 
\begin{equation}\label{eq:20}
K\gtrsim|\log\delta|^{O(1)}\delta^{O(\epsilon)} \delta^{-\frac{5\gamma-12\beta+\eta-\kappa(2-\gamma)}{16m_2+44m_1}}.
\end{equation}
Using inequality (\ref{13}) instead, one has
\begin{equation}\label{eq:21}
K\gtrsim |\log\delta|^{O(1)}\delta^{O(\epsilon)}\delta^{-\frac{6\gamma-13\beta+\eta-\kappa(1-\eta)}{16m_2+44m_1}}.
\end{equation}
\subsection{Gap case}
In this section, we will suppose that we are in the gap case $(A)$. This means that there exists $r\in R\cap [0,1]$ so that either 
\begin{itemize}
    \item[(A.1)] $r/2$ is at least $s$-separated from $R$ or
    \item[(A.2)] $\frac{r+1}{2}$ is at least $s$-separated from $R$, where $s \sim \delta^{1-2\kappa}|b_2-b_3|$.
\end{itemize}
Notice that it follows from the definition of $\widetilde{D}$ that 
$r=\frac{c_1-c_2}{c_3-c_4}$
for some $c_1, c_2, c_3, c_4\in C'$ with $|c_3-c_4|\ge |c_1-c_2|$ and  $|c_3-c_4|\ge \delta^{\kappa}=\delta^{\epsilon_0}$.

In Case (A.1), we write $r/2=e_1/e_2$ with $e_1=x_1$ and $e_2=x_2+x_2$, where $x_1, x_2\in C'-C'$. 

In Case (A.2), we write $\frac{r+1}{2}$ as $e_1/e_2$ with $e_1 = x_1 + x_2,\ e_2 = x_2+x_2$, where $x_1,x_2\in C'-C'$. 

The first task in this section is to find a lower bound on $|e_1 \widetilde{D} + e_2 \widetilde{D}|_\delta$. One needs to keep in mind that  $|e_2|\sim |c_3-c_4|$. 

Define $Q \subset \widetilde{D}^4$ to be the set of quadruples obeying
\begin{equation} \label{quads} 
e_2 d_1 + e_1 d_4 = e_2 d_2 + e_1 d_3 + O(\delta). 
\end{equation}
As in the dense case, we only need to find an upper bound for $|Q|$. Since by the Cauchy-Schwarz inequality, one has $|e_1 \widetilde{D} + e_2 \widetilde{D}|_\delta \gtrsim \frac{|\widetilde{D}|}{|Q|}$. 

Dividing equation (\ref{quads}) by $e_2$ gives
\begin{equation}\label{farFromD1D2}
\left| \frac{d_1 - d_2}{d_3 - d_4} - \frac{e_1}{e_2} \right| \lesssim \delta |e_2|^{-1} |d_3 - d_4|^{-1}.
\end{equation}

Assume that there exists a quadruple $(d_1, ..., d_4) \in Q$ such that $|d_3 - d_4| \ge \delta^\kappa$. In other words, we have $ \frac{d_1 - d_2}{d_3 - d_4} \in R$. Then using the fact that $|b_2-b_3||e_2|\gtrsim \delta^{\kappa}$, equation \eqref{farFromD1D2} implies
$$ \left| \frac{d_1 - d_2}{d_3 - d_4} - \frac{e_1}{e_2} \right| \lesssim \delta^{1 - 2 \kappa}{|b_2-b_3|}. $$
Since we are in the gap case, $r=e_1/e_2$ is at least $s \sim \delta^{1-2\kappa}|b_2-b_3|$ separated from $R$, which is a contradiction. It turns out that every quadruple in $Q$ satisfies $|d_3 - d_4| \le \delta^\kappa$.

As in the dense case, by using the non-concentration assumptions on $C$ and $B$, respectively, we obtain the following bounds.

{(i) Non-concentration on $C$:}
\begin{equation}\label{eq19}
|e_1 \widetilde{D} + e_2\widetilde{D}|_\delta \gtrsim \delta^{O(\epsilon)} \frac{|C||\log\delta|^{-2}|b_2-b_3|}{K^{8m_1}|B|^2}.
\end{equation}

{(ii) Non-concentration on $B$:}
\begin{equation}\label{20}
|e_1 \widetilde{D} + e_2\widetilde{D}|_\delta \gtrsim \delta^{O(\epsilon)}\frac{|C|^2|\log\delta|^{-2}}{|B|^4K^{8m_1}|B|^2\delta^{\eta}}.
\end{equation}

In the next step, we apply Lemma \ref{lem3.3} to get upper bounds on $|e_1 \widetilde{D} + e_2\widetilde{D}|_\delta$. Then all possibilities for bounds of $K$ will be examined.

\begin{enumerate}
    \item[(a)] {First upper bound}
    
    Recall that $|e_1 \widetilde{D} + e_2\widetilde{D}|_\delta\lesssim |x_1\widetilde{D}+x_2\widetilde{D}+x_2\widetilde{D}+x_2\widetilde{D}|_\delta$.
Applying Lemma \ref{lem3.3} (i) with $k=3$, one has 
\[
|x_1\widetilde{D}+x_2\widetilde{D}+x_2\widetilde{D}+x_2\widetilde{D}|_\delta\lesssim \delta^{-O(\epsilon)}\cdot \frac{|B|^{12} K^{20m_2+46m_1}|\log\delta|^6|e_2|\delta^{\epsilon_0(\gamma-1)}}{|b_2-b_3| |C|^5}.
\]
Then one can combine with the lower bound (\ref{eq19}) to get
\begin{equation}\label{eq:27}
K\gtrsim |\log\delta|^{O(1)}\delta^{O(\epsilon)}\delta^{-\frac{6\gamma-14\beta -\kappa(3-\gamma)}{20m_2+54m_1}}.
\end{equation}
Similarly, replacing \eqref{eq19} with the lower bound \eqref{20}, one obtains
\begin{equation}\label{eq:28}
K\gtrsim |\log\delta|^{O(1)}\delta^{O(\epsilon)}\delta^{-\frac{7\gamma-16\beta+\eta-\kappa(2-\gamma)}{20m_2+54m_1}}.
\end{equation}

 \item[(b)] {Second upper bound }
    
Applying Lemma \ref{lem3.3}(ii) with $k=3$, one has 
\[ 
|x_1\widetilde{D}+x_2\widetilde{D}+x_2\widetilde{D}+x_2\widetilde{D}|_\delta\lesssim \delta^{-O(\epsilon)}\cdot \frac{|B|^{13}K^{20m_2+46m_1}|\log\delta|^6|e_2|\delta^{\kappa(\eta-1)}}{|C|^6}.
\]
Incorporating with the lower bounds (\ref{eq19}) and (\ref{20}), we have
\begin{equation}\label{eq:29}
 K\gtrsim |\log\delta|^{O(1)}\delta^{O(\epsilon)}\delta^{-\frac{7\gamma-15\beta-\kappa(2-\eta)}{20m_2+54m_1}},
 \end{equation}
 and
\begin{equation}\label{eq:30}
K\gtrsim |\log\delta|^{O(1)}\delta^{O(\epsilon)}\delta^{-\frac{8\gamma-17\beta+\eta-\kappa(1-\eta)}{20m_2+54m_1}}\,.
\end{equation}
\end{enumerate}
\subsection{Concluding the proof}
Let us summarize the lower bounds for $K$ here:
\begin{itemize}
    \item (\ref{eq:16}) $ K\gtrsim  |\log\delta|^{O(1)}\delta^{O(\epsilon)}
    \delta^{-\frac{3\gamma-8\beta+1-\kappa(3-\gamma)}{16m_2+36m_1}}$,
    \item (\ref{eq:17}) $ K\gtrsim |\log\delta|^{O(1)}\delta^{O(\epsilon)}\delta^{-\frac{4\gamma-9\beta+1-\kappa(3-\eta)}{16m_2+36m_1}}$,
    \item (\ref{eq:18}) $ K \gtrsim |\log\delta|^{O(1)}\delta^{O(\epsilon)}\delta^{-\frac{4\gamma-10\beta-\kappa(3-\gamma)}{16m_2+44m_1}} $,
    \item (\ref{eq:19}) $K \gtrsim |\log\delta|^{O(1)}\delta^{O(\epsilon)}\delta^{-\frac{5\gamma-11\beta-\kappa(2-\eta)}{16m_2+44m_1}} $,
    \item (\ref{eq:20}) $K \gtrsim |\log\delta|^{O(1)}\delta^{O(\epsilon)}\delta^{-\frac{5\gamma-12\beta+\eta-\kappa(2-\gamma)}{16m_2+44m_1}} $,
    \item (\ref{eq:21}) $K \gtrsim |\log\delta|^{O(1)}\delta^{O(\epsilon)}\delta^{-\frac{6\gamma-13\beta+\eta-\kappa(1-\eta)}{16m_2+44m_1}} $,
    \item (\ref{eq:27}) $ K \gtrsim |\log\delta|^{O(1)}\delta^{O(\epsilon)}\delta^{-\frac{6\gamma-14\beta-\kappa(3-\gamma)}{20m_2+54m_1}} $,
    \item (\ref{eq:28}) $ K\gtrsim |\log\delta|^{O(1)}\delta^{O(\epsilon)}\delta^{-\frac{7\gamma-16\beta+\eta-\kappa(2-\gamma)}{20m_2+54m_1}} $,
    \item (\ref{eq:29}) $ K\gtrsim |\log\delta|^{O(1)}\delta^{O(\epsilon)}\delta^{-\frac{7\gamma-15\beta-\kappa(2-\eta)}{20m_2+54m_1}} $,
    \item (\ref{eq:30}) $ K \gtrsim |\log\delta|^{O(1)}\delta^{O(\epsilon)}\delta^{-\frac{8\gamma-17\beta+\eta -\kappa(1-\eta)}{20m_2+54m_1}} $.
\end{itemize}
where $\kappa$ can be chosen arbitrarily in $(0, 1/2)$. 

Now we put $M_0=\delta^{-\frac{\gamma-\beta}{4}}$ and $M_1= \delta^{-\frac{\gamma+2\beta+\kappa\eta}{4}}$, which are given by \eqref{m0} and \eqref{eq_K_nonempty_R}, respectively. Then define 
\begin{align*}
    &M_2=\max\big\{ \eqref{eq:16},\eqref{eq:17}\big\} ,\\
    &M_3=\max\big\{ \eqref{eq:18},\eqref{eq:19},\eqref{eq:20},\eqref{eq:21}\big\},\\
    &M_4=\max\big\{ \eqref{eq:27},\eqref{eq:28},\eqref{eq:29},\eqref{eq:30}\big\}.
\end{align*}
Following the proof, one can see that $K$ can be bounded from below by
\begin{align*}
    K\gtrsim |\log\delta|^{O(1)}\delta^{O(\epsilon)} \min \{M_0, M_1, M_2, M_3, M_4\}.
\end{align*}
We complete the proof of Theorem \ref{thm_energy_III}.
\section{Proof of Theorem \ref{thm-G1l} and Theorem \ref{thm-G2}}

By the Cauchy-Schwarz inequality and Theorem \ref{thm_energy_III}, it can be seen easily that there exists $c\in C$ such that 
\[|A+cB|_\delta\gtrsim \frac{K|B|^{1/2}}{|A|^{1/2}}|A|.\]
Therefore, to show $|A+cB|_\delta\ge \delta^{-\epsilon'}|A|$ for some $\epsilon'>0$, we need 
\begin{align}\label{eq_KAB_deltaepsilon}
    \frac{K|B|^{1/2}}{|A|^{1/2}}\gtrsim \delta^{-\epsilon'}.
\end{align}
On the other hand, from Theorem \ref{thm_energy_III}, we have
\begin{align*}
     K\gtrsim |\log\delta|^{O(1)}\delta^{O(\epsilon)} \min \{M_0, M_1, M_2, M_3, M_4\}.
\end{align*}
We now compute the ranges of $\gamma$ corresponding to the above cases such that \eqref{eq_KAB_deltaepsilon} holds, namely,

\begin{itemize}
\item \eqref{m0} $\gamma>2\alpha-\beta$.
\item \eqref{eq_K_nonempty_R} $\gamma>2\alpha-4\beta-\kappa\eta$.
    \item \eqref{eq:16}
           $ \gamma >\frac{74\alpha-66\beta-1+3\kappa}{3+\kappa}$.
    \item \eqref{eq:17}
    $ \gamma >\frac{74\alpha-65\beta-1+\kappa(3-\eta)}{4}.    $
    \item \eqref{eq:18}
    $        \gamma >\frac{78\alpha-68\beta+3\kappa}{4+\kappa}.$
    \item \eqref{eq:19}
    $        \gamma >\frac{78\alpha-67\beta+\kappa(2-\eta)}{5}.$
    \item \eqref{eq:20}
    $       \gamma >\frac{78\alpha-66\beta-\eta+2\kappa}{5+\kappa}.$
    \item \eqref{eq:21}
    $\gamma >\frac{78\alpha-65\beta-\eta+\kappa(1-\eta)}{6}.$
    \item \eqref{eq:27}
    $\gamma >\frac{97\alpha-83\beta+3\kappa}{6+\kappa}.$
    \item \eqref{eq:28}
    $\gamma >\frac{97\alpha-81\beta-\eta +2\kappa}{7+\kappa}.$
    \item \eqref{eq:29}
    $\gamma >\frac{97\alpha-82\beta+\kappa(2-\eta)}{7}.$
    \item \eqref{eq:30}
    $\gamma >\frac{97\alpha-80\beta-\eta+\kappa(1-\eta)}{8}.$
\end{itemize}
Set $\eta=\beta$ and choose $\kappa=\kappa(\alpha, \beta)$ close to zero. 

Among (\ref{eq:18}, \ref{eq:19}, \ref{eq:20}, \ref{eq:21}), the widest range for $\gamma$ is 
\[\gamma>\frac{78\alpha-66\beta}{6},\]
which comes from (\ref{eq:21}).

Among (\ref{eq:27}, \ref{eq:28}, \ref{eq:29}, \ref{eq:30}), the widest range for $\gamma$ is 
\[\gamma>\frac{97\alpha-81\beta}{8},\]
which comes from (\ref{eq:30}).

On the other hand, since $\alpha\ge \beta$, one has 
\[\frac{78\alpha-66\beta}{6}>\frac{97\alpha-81\beta}{8}.\]

This means that we end up with the following two ranges:

\begin{enumerate}
    \item The case (\ref{m0}, \ref{eq_K_nonempty_R}, \ref{eq:16}, \ref{eq:21}, \ref{eq:30}):
    \[\gamma>\max\left\lbrace 2\alpha-\beta, \frac{74\alpha-66\beta-1}{3}, \frac{78\alpha-66\beta}{6} \right\rbrace.\]
Using the fact $\alpha\ge \beta$, we conclude that 

If $\alpha<\frac{33}{35}\beta+\frac{1}{35}$, then 
\[\gamma>\frac{78\alpha-66\beta}{6}.\]
If $\alpha>\frac{33}{35}\beta+\frac{1}{35}$, then 
\[\gamma>\frac{74\alpha-66\beta-1}{3}.\]

    \item The case (\ref{m0}, \ref{eq_K_nonempty_R}, \ref{eq:17}, \ref{eq:21}, \ref{eq:30}):
    \[\gamma>\max\left\lbrace 2\alpha-\beta, \frac{74\alpha-65\beta-1}{4}, \frac{78\alpha-66\beta}{6} \right\rbrace.\]

If $\alpha>\frac{63\beta}{66}+\frac{1}{22}$, then 
\[\gamma>\frac{74\alpha-65\beta-1}{4}.\]
If $\alpha<\frac{63\beta}{66}+\frac{1}{22}$, then 
\[\gamma>\frac{78\alpha-66\beta}{6}.\]
\end{enumerate}

Comparing between the two cases, we infer that (\ref{m0}, \ref{eq_K_nonempty_R}, \ref{eq:17}, \ref{eq:21}, \ref{eq:30}) gives the best range for $\gamma$. Namely,

If $\alpha\le \frac{63}{66}\beta+\frac{1}{22}$, then the condition 
\[\gamma>\frac{78\alpha-66\beta}{6}\]
would be enough to have (\ref{eq_KAB_deltaepsilon}). Moreover, we also have \[K\gtrsim \min \left\lbrace \delta^{-\frac{4\gamma-9\beta+1}{148}}, ~\delta^{-\frac{6\gamma-12\beta}{156}} \right\rbrace.\]

If $\alpha>\frac{63}{66}\beta+\frac{1}{22}$, then the condition 
\[\gamma>\frac{74\alpha-65\beta-1}{4}\]
would be enough to have (\ref{eq_KAB_deltaepsilon}). Similarly, we also have \[K\gtrsim \min \left\lbrace \delta^{-\frac{4\gamma-9\beta+1}{148}}, ~\delta^{-\frac{6\gamma-12\beta}{156}} \right\rbrace.\]
With these computations, we are ready to prove Theorem \ref{thm-G1l} and Theorem \ref{thm-G2}. Let us first recall the statements.

\begin{theorem}[Theorem \ref{thm-G1l}]
 Let $\alpha, \beta, \eta \in (0,1)$, $\beta\le \alpha\le (21\beta+1)/22$. Then, for every $\gamma\in ((78\alpha-66\beta)/6, 1]$, there exist $\epsilon_0, \delta_0\in (0, 1/2]$, depending only on $\alpha, \beta, \gamma, \eta$, such that the following holds. Let $\delta\in 2^{-\mathbb{N}}$ with $\delta\in (0, \delta_0]$, and let $A, B\subset [0, 1]$ be $\delta$-separated sets satisfying the following hypotheses: 
\begin{itemize}
    \item $|A|\le  \delta^{-\alpha}$
    \item $|B|\ge \delta^{-\beta}$, and $B$ satisfies the following Frostman condition 
    \[|B\cap B(x, r)|\le r^\eta ~ |B|, ~~ \forall x\in \R, \delta\le r\le \delta^{\epsilon_0}.\]
\end{itemize}
Further, let $C\subset [1/2, 1]$ be a $\delta$-separated set with $|C\cap B(x, r)|\lesssim r^\gamma~ |C|$ for all $x\in \mathbb{R}$ and $0<r<\delta^{\epsilon_0}$. Then, there exists $c\in C$ such that the following holds for any $\epsilon$ satisfying
\[0<\epsilon<\min \left\lbrace \frac{4\gamma-74\alpha+65\beta+1}{444}, ~\frac{6\gamma-78\alpha+66\beta}{468} \right\rbrace.\]
If $G\subset A\times B$ is any subset with $|G|\ge \delta^\epsilon |A||B|$, then  
\[|\pi_c(G)|_\delta\ge \delta^{-\epsilon}|A|.\]
\end{theorem}
\begin{theorem}[Theorem \ref{thm-G2}]
 Let $\alpha, \beta, \eta \in (0,1)$, $\beta\le \alpha$ and $\alpha>(21\beta+1)/22$. Then, for every $\gamma\in ((74\alpha-65\beta-1)/4, 1]$, there exist $\epsilon_0, \delta_0\in (0, 1/2]$, depending only on $\alpha, \beta, \gamma, \eta$, such that the following holds. Let $\delta\in 2^{-\mathbb{N}}$ with $\delta\in (0, \delta_0]$, and let $A, B\subset [0, 1]$ be $\delta$-separated sets satisfying the following hypotheses: 
\begin{itemize}
    \item $|A|\le  \delta^{-\alpha}$
    \item $|B|\ge \delta^{-\beta}$, and $B$ satisfies the following Frostman condition 
    \[|B\cap B(x, r)|\le r^\eta ~ |B|, ~~ \forall x\in \R, \delta\le r\le \delta^{\epsilon_0}.\]
\end{itemize}
Further, let $C$ be a $\delta$-separated set in $[1/2, 1]$ with $|C\cap B(x, r)|\lesssim r^\gamma~ |C|$ for all $x\in \mathbb{R}$ and $0<r<\delta^{\epsilon_0}$. Then, there exists $c\in C$ such that the following holds for any $\epsilon$ satisfying
\[0<\epsilon<\min \left\lbrace \frac{4\gamma-74\alpha+65\beta+1}{444}, ~\frac{6\gamma-78\alpha+66\beta}{468} \right\rbrace.\]
If $G\subset A\times B$ is any subset with $|G|\ge \delta^\epsilon |A||B|$, then  
\[|\pi_c(G)|_\delta\ge \delta^{-\epsilon}|A|.\]
\end{theorem}
The proofs of these two theorems are the same, so we only present a proof for the first one. Again, the proof is short and follows directly from the energy estimates. 
\begin{proof}[Proof of Theorem \ref{thm-G1l}]
Let $X$ be the set of $c\in C$ such that the conclusion of the theorem fails, i.e. for each $c\in X$, there exists $G_c\subset A\times B$ with $|G_c|\ge \delta^\epsilon |A||B|$ and $|\pi_c(G)|_\delta< \delta^{-\epsilon}|A|$. We want to show that $|X|<|C|$.

For each $c\in X$, by the Cauchy-Schwarz inequality, the number of tuples $(x, y, z, w)\in G_c^2$ such that $|(x+cy)-(z+cw)|\le \delta$ is at least 
$\frac{|G_c|^2}{\delta^{-\epsilon}|A|}$, which equals $\delta^{3\epsilon}|A||B|^2$.
Summing over all $c\in X$ and using Theorem \ref{thm_energy_III}, one has 
\begin{align*}
    |X|\delta^{3\epsilon}|A||B|^2\le \frac{|C||A|^{3/2}|B|^{3/2}}{K}.
\end{align*}
This infers
\[|X|\le \delta^{-3\epsilon}\frac{|A|^{1/2}}{|B|^{1/2}K}|C|.\]
Using the computations above, we know that 
\[\frac{K|B|^{1/2}}{|A|^{1/2}}\gtrsim \min \left\lbrace \delta^{-\frac{4\gamma-74\alpha+65\beta+1}{148}}, ~\delta^{-\frac{6\gamma-78\alpha+66\beta}{156}} \right\rbrace.\]
Thus, 
\[\delta^{-3\epsilon}\frac{|A|^{1/2}}{|B|^{1/2}K}\lesssim \delta<1\]
as long as 
\[\epsilon<\min \left\lbrace \frac{4\gamma-74\alpha+65\beta+1}{444}, ~\frac{6\gamma-78\alpha+66\beta}{468} \right\rbrace.\]
This completes the proof of the theorem.
\end{proof}

\section{Proof of Theorem \ref{thm_energy_II} and Theorem \ref{a+cba41}}
As mentioned in the introduction, to prove Theorem \ref{thm_energy_II}, we make use of the following point-tube incidence bound due to Dabrowski, Orponen, and Villa in \cite{DOV}. 

\begin{theorem}\label{t:incidences} Let $0 < n < d$ and $M\geq 1$. Let $\mathcal{V} \subset \mathcal{A}(d,n)$ be a $\delta$-separated set of $n$-planes, and let $P \subset B(1) \subset \R^{d}$ be a $\delta$-separated $(\delta,t, M)$-set with $t > d - n$, i.e.
\[|P\cap B(x, r)|_\delta \le M r^t|P|_\delta, ~\forall  ~x\in \mathbb{R}^d, ~r>0.\]
For $r>0$, define $\mathcal{I}_{r}(P,\mathcal{V}) = \{(p,V) \in P \times \mathcal{V} : p \in V(r)\}$. Then we have
	\begin{equation*}
		|\mathcal{I}_{M\delta}(P,\mathcal{V})| \lesssim_{M, d,t}  |P| \cdot |\mathcal{V}|^{n/(d+n-t)} \cdot \delta^{n(t+1-d)(d-n)/(d+n-t)}.
	\end{equation*}
\end{theorem}

Here, the Grassmannian $\mathcal{A}(d, n)$ is equipped with the metric $d_\mathcal{A}$ defined as follows. For $V, W\in \mathcal{A}(d, n)$, let $V_0$, $W_0$, and $a\in V_0^\perp, b\in W_0^\perp$ be unique subspaces and vectors such that 
\[V=V_0+a, ~W=W_0+b.\]
The distance between $V$ and $W$ is defined by 
\[d_{\mathcal{A}}(V, W)=||\pi_{V_0}-\pi_{W_0}||_{op}+|a-b|,\]
where $||\cdot||_{op}$ is the operator norm.

In the plane, a direct computation shows that the distance between two lines $y=ax+b$ and $y=cx+d$ is 
\[\left\vert \frac{(a, -1)}{|(a, -1)|}-\frac{(c, -1)}{|(c, -1)|}\right\vert+\left\vert \frac{b}{|(a, -1)|}-\frac{d}{|(c, -1)|} \right\vert.\]

\begin{proof}[Proof of Theorem \ref{thm_energy_II}]
The first part is identical with Theorem \ref{thm_energy_I}. For the second part, since $A, B, C$ are $\delta$-separated, the number of tuples $(a_1, a_2, b_1, b_2, c)\in A^2 \times B^2\times C$ such that \[|(a_1+cb_1)-(a_2+cb_2)|\le \delta\]
is $|A|^{3/2}|B|^{3/2}|C|K^{-1}$. 

For a fixed $a_2$, the expression $|(a_1-a_2)-c(b_2-b_1)|\le \delta$ infers that the point $(b_2, a_1-a_2)$ belongs to the $\delta$-neighborhood of the line defined by $y=c(x-b_1)$. Let $L$ be the set of such lines.

Since $C$ and $B$ are $\delta$-separated, one can directly check using the above metric to get the set $L$ is $c_0\delta$ separated for some absolute constant $c_0>0$.

Set $P=B\times (A-a_2)$. Since $B$ and $A$ are $\delta$-separated and 
\[|B\cap B(x, r)|\le Mr^\beta|B|,\]
and 
\[|A\cap B(x, r)|\le Mr^\alpha|A|,\]
for all $\delta\le r\le 1$. Therefore we have 
\[|P\cap B(x, r)|\le M^2r^{\alpha+\beta}|P|.\]

If $\alpha+\beta>1$, then we can apply Theorem \ref{t:incidences} to obtain 
\[|I_{M\delta}(P, L)\lesssim |P||L|^{\frac{1}{3-\alpha-\beta}}\delta^{\frac{\alpha+\beta-1}{3-\alpha-\beta}}=|A||B|(|B||C|)^{\frac{1}{3-\alpha-\beta}}\delta^{\frac{\alpha+\beta-1}{3-\alpha-\beta}}.\]
Summing over all $a_2\in A$, we obtain 
\[|A|^{3/2}|B|^{3/2}|C|K^{-1}\lesssim |A|^2|B|(|B||C|)^{\frac{1}{3-\alpha-\beta}}\delta^{\frac{\alpha+\beta-1}{3-\alpha-\beta}}.\]
This gives 
\[K\gtrsim \delta^{\frac{\alpha-3\beta-4\gamma+2\gamma(\alpha+\beta)-\alpha^2+\beta^2+2}{2(3-\alpha-\beta)}},\]
concluding the proof.
\end{proof}
Theorem \ref{a+cba41} follows immediately from Theorem \ref{thm_energy_II} and the Cauchy-Schwarz inequality. So we omit the details.
\section{Acknowledgements}
We would like to thank Toumas Orponen for a number of comments on the first draft of this paper. 

T. Pham would like to
thank to the VIASM for the hospitality and for the excellent working conditions. 

T. Pham was supported by Vietnam National Foundation for Science and Technology Development (NAFOSTED) under grant number 101.99--2021.09.
C-Y Shen was partially supported by NSTC grant 111-2115-M-002-010-MY5. 

\textbf{Declaration of interests}: The authors do not work for, advise, own shares in, or receive funds from any organisation that could benefit from this article, and have declared no affiliation other than their research organisations

\end{document}